\documentclass[11pt]{article}

\usepackage[T1]{fontenc}
\usepackage{lmodern}
\usepackage{amsmath,amssymb,mathtools}
\usepackage{amsthm}
\usepackage{enumitem}
\usepackage{microtype}
\usepackage[margin=1.12in]{geometry}
\usepackage{xcolor}
\usepackage[colorlinks=true,linkcolor=blue!55!black,
  citecolor=blue!55!black,urlcolor=blue!60!black]{hyperref}

\numberwithin{equation}{section}
\mathtoolsset{showonlyrefs=true} 
\newtheorem{theorem}{Theorem}[section]
\newtheorem{proposition}[theorem]{Proposition}
\newtheorem{lemma}[theorem]{Lemma}
\newtheorem{corollary}[theorem]{Corollary}
\theoremstyle{definition}

\theoremstyle{remark}
\newtheorem{remark}[theorem]{Remark}

\newcommand{\Om}{\Omega}
\newcommand{\D}{\mathbb D}
\newcommand{\dv}{\,dv_g}

\newcommand{\dd}{\,d}

\newcommand{\Span}{\operatorname{span}}

\title{{ Sharp Harmonic-Mean Bounds for Neumann Eigenvalues\\ on Riemannian Surfaces\\and Counterexamples to Geodesic Ball Optimality\\
on Higher-Dimensional Spheres}}
\author{Hongtao Hu\thanks{\href{mailto:huht23@mails.tsinghua.edu.cn}{\texttt{huht23@mails.tsinghua.edu.cn}}} \and Meiqi Liu\thanks{\href{mailto:liumq@hdu.edu.cn}{\texttt{liumq@hdu.edu.cn}}} \and Qiaoran Wu\thanks{\href{mailto:wuqr24@mails.tsinghua.edu.cn}{\texttt{wuqr24@mails.tsinghua.edu.cn}}} \and Wenming Zou\thanks{\href{mailto:zou-wm@math.tsinghua.edu.cn}{\texttt{zou-wm@math.tsinghua.edu.cn}}}}
\date{}

\begin{document}

\maketitle

\begin{abstract}
{
\noindent Let $\Omega\subset\mathbb C$ be a bounded simply connected Lipschitz domain with conformal metric $g=\omega|dz|^2$, where $\omega\in C^2(\Omega)\cap L^\infty(\Omega)$ is positive in the interior and may tend to zero at the boundary. If the Gaussian curvature satisfies $K_g\leq K$, $K>0$, and $KM<4\pi$, where $M=\int_\Omega\omega\,dA$, then the Neumann eigenvalues, numbered by $0=\mu_1(\Omega;\omega)<\mu_2(\Omega;\omega)\leq\mu_3(\Omega;\omega)\leq\cdots$, satisfy
\[
 \frac1{\mu_2(\Omega;\omega)}+\frac1{\mu_3(\Omega;\omega)}
 \geq\frac2{\mu_2(D_K(M))}.
\]
Here $D_K(M)$ is the geodesic disk of area $M$ in the surface of constant curvature $K$. We prove that equality holds exactly when the interiors are isometric. This completes the Conjecture~1.2 of Langford--Laugesen; in particular, the model disk maximizes the first positive Neumann eigenvalue in this class.}

{In contrast, geodesic balls need not maximize the first positive Neumann eigenvalue at fixed volume on higher-dimensional spheres, even among smooth simply connected domains.} Indeed, for every fixed dimension $n\ge3$, we construct a smooth simply connected domain $\Om$ in the unit sphere $\mathbb S^n$ whose first positive Neumann eigenvalue is strictly larger than that of a { geodesic ball} $B$ of the same volume:
\[
 \mu_2(\Om)>\mu_2(B),\qquad |\Om|=|B|.
\]
Together with the known two-dimensional result, this completely disproves Conjecture~1.3 of Langford--Laugesen.
\end{abstract}



\section{Introduction}
\enlargethispage{2pt}

Let $(\Om,g)$ be a compact connected Riemannian surface with boundary. {The Laplace--Beltrami operator is defined by}
\[
 {\Delta_g u:=\operatorname{div}_g(\nabla_g u),}
\]
where $\nabla_g$ and $\operatorname{div}_g$ denote the gradient and divergence with respect to $g$, {respectively}. We write $\Delta=\partial_x^2+\partial_y^2$ for the classical Euclidean Laplacian, so that $\Delta_g=\Delta$ for the Euclidean metric.

In the smooth setting, the Neumann eigenvalues in this paper are the eigenvalues of $-\Delta_g$ in the boundary-value problem
\[
 \begin{cases}
  -\Delta_g u=\mu u,&\text{in $\Om$},\\
  \partial_{\nu_g}u=0,&\text{on $\partial\Om$},
 \end{cases}
 \qquad u\not\equiv0.
\]
Here $\nu_g$ is the outward unit normal and $\partial_{\nu_g}u=\langle\nabla_g u,\nu_g\rangle_g$. In the Lipschitz setting the eigenvalue problem is understood in the weak sense. We arrange the eigenvalues in nondecreasing order, counted with multiplicity,
\[
  0=\mu_1(\Om,g)<\mu_2(\Om,g)\leq\mu_3(\Om,g)\leq\cdots.
\]
For $K\in\mathbb R$ and $M>0$ (with $M<4\pi/K$ when $K>0$), let $D_K(M)$ denote a geodesic disk of area $M$ in the complete simply connected surface of constant curvature $K$. Its first positive Neumann eigenvalue has multiplicity two:
\begin{equation}\label{eq:model-double}
  \mu_2(D_K(M))=\mu_3(D_K(M)).
\end{equation}

Langford--Laugesen \cite[Theorem~1.1]{LangfordLaugesen2023} established a sharp harmonic-mean comparison on simply connected Lipschitz surfaces. If $K_g\leq K$ and $M=|\Om|_g$, they proved
\[
  \frac1{\mu_2(\Om,g)}+\frac1{\mu_3(\Om,g)}
  \geq\frac2{\mu_2(D_K(M))}
  \qquad\text{when}\qquad KM\leq\frac{16}{17}\,4\pi.
\]
Equivalently, the harmonic mean of the first two positive eigenvalues {does not exceed} the first positive eigenvalue of the equal-area model disk. {They also proved equality rigidity: equality holds if and only if $(\Omega,g)$ is isometric to $D_K(M)$.} For $K\leq0$ the area condition is automatic. Their new contribution concerned $K>0$, where they extended Bandle's range $KM\leq2\pi$ \cite{Bandle1972} to $KM\leq(16/17)4\pi$. Thus, among simply connected Lipschitz domains of prescribed area on the sphere, a spherical cap maximizes the first positive Neumann eigenvalue whenever the prescribed area is at most $94.1\%$ of the area of the sphere.

The threshold $16/17$ reflects the reach of that method and need not be the optimal range for the harmonic-mean conclusion. Langford--Laugesen therefore posed the following open problem.

\smallskip
\noindent\textbf{Conjecture 1.2 (Area up to $4\pi$).} If $(\Om,g)$ is a simply connected Lipschitz surface satisfying $K_g\leq K$ with $K>0$, and $M=|\Om|_g$ satisfies $KM<4\pi$, then
\begin{equation}\label{eq:LL-range}
  \frac1{\mu_2(\Om,g)}+\frac1{\mu_3(\Om,g)}
  \geq \frac2{\mu_2(D_K(M))}.
\end{equation}
In particular,
    \begin{equation}\label{eq:second-eigenvalue}
    	\mu_2(\Omega,g)\leq\mu_2(D_K(M)).
    \end{equation}

The condition $KM<4\pi$ is natural because $4\pi/K$ is the area of the complete sphere of constant curvature $K$, and the model disk {attains} equality in \eqref{eq:LL-range}.

Langford--Laugesen also proposed the following higher-dimensional conjecture.
    
\smallskip
\noindent\textbf{Conjecture~1.3 (Higher dimensions).} If $\Om$ is a Lipschitz subdomain of the unit sphere $\mathbb S^n$ ($n\geq2$), then
\begin{equation}\label{eq:LL-higher-dimensional}
  \mu_2(\Om)\leq\mu_2(B),\qquad |B|=|\Om|,
\end{equation}
where $B$ is a geodesic ball.

Here $|\cdot|$ denotes $n$-dimensional {measure on $\mathbb{S}^n$.} Conjecture~1.3 assumes neither simple connectedness nor a volume restriction. Langford--Laugesen also raised the possibility of a weaker version in which $|\Om|$ is at most half the volume of the sphere.

For Conjecture~1.2, Provenzano--Savo \cite[Theorem~1.3]{ProvenzanoSavo2026} established on smooth simply connected surfaces the single-eigenvalue bound
    \begin{equation}\label{eq:single-value-comparison}
    	\mu_2(\Om,g)\leq\mu_2(D_K(M)),\quad KM<4\pi,
    \end{equation}
which is \eqref{eq:second-eigenvalue}. {For spherical domains, they also proved equality rigidity \cite[Theorem~1.1]{ProvenzanoSavo2026}. The harmonic-mean bound \eqref{eq:LL-range} is stronger than \eqref{eq:single-value-comparison}, so their result gives a partial answer to Conjecture~1.2.}

{Chen--Yang \cite[Theorem~1.1 and Corollary~3.2]{ChenYang2026} proved the sharp harmonic-mean inequality for smooth simply connected spherical domains, with equality only for geodesic disks, and extended the inequality to smooth simply connected surfaces with a Gaussian curvature upper bound. Their Corollary~3.2 also includes $K\leq0$ and the endpoint $KM=4\pi$ when $K>0$. On the common smooth range $K>0$ and $KM<4\pi$, their bound coincides with \eqref{eq:LL-range}.}

{We take the smooth harmonic-mean comparison from \cite[Corollary~3.2]{ChenYang2026} as a preliminary result. We retain only the Green-radial quantities and the intermediate estimate at a balanced pole that are needed for the Lipschitz equality argument. These come from the framework of Provenzano--Savo \cite{ProvenzanoSavo2026} and its two-dimensional variational application in \cite[Section~3]{ChenYang2026}.}

{Compared with the smooth results of Chen--Yang, Theorem~\ref{thm:lipschitz-main} below treats Lipschitz domains with bounded weights that are positive in the interior but may tend to zero at the boundary, and includes an equality characterization by interior isometry.}

{For Conjecture 1.3, Bucur et al. \cite[Theorem~1]{BucurLaugesenMartinetNahon2025} constructed a counterexample on $\mathbb{S}^2$: there exists $\delta>0$ such that for each $\theta\in(\Theta-\delta,\pi)$, there exists an open set $\Omega_{\theta}\subset\mathbb{S}^2$ with $\lvert\Omega_{\theta}\rvert=\lvert B_{\theta}\rvert$ and
    \[\mu_{2}(\Omega_{\theta})>\mu_{2}(B_{\theta}),\]
where $\Theta\approx 0.70\pi$ is the unique constant such that $\mu_{2}(B_{\Theta})\sin^2\Theta=1$, and $B_{\theta}$ is the spherical cap with aperture angle $\theta$. More precisely, $\Omega_{\theta}$ is obtained by removing four small holes from $B_{\theta}$ at the same latitude, followed by a boundary perturbation to preserve the area; see \cite[Section~4]{BucurLaugesenMartinetNahon2025}. The first positive eigenvalue is denoted by $\mu_1$ in \cite{BucurLaugesenMartinetNahon2025} and by $\mu_2$ here. Later, Provenzano--Savo \cite[Theorem~1.2]{ProvenzanoSavo2026} found another annular counterexample: for the symmetric spherical annulus $\Omega_\varepsilon=\{(r,\theta):\varepsilon<r<\pi-\varepsilon\}$ and the equal-area cap $B_\varepsilon$, one has $\mu_2(\Omega_\varepsilon)>\mu_2(B_\varepsilon)$ when $\varepsilon>0$ is sufficiently small. Thus Conjecture 1.3 fails for dimension $2$.

In this paper, we adopt the hole-cutting idea from \cite{BucurLaugesenMartinetNahon2025} to give counterexamples for arbitrary dimension \(n\ge 3\). Our construction also requires the volume to be greater than $\lvert\mathbb{S}^n\rvert/2$. Thus, whether that weaker version holds remains open.}

{We now state our results, beginning with the Lipschitz extension. Smooth-boundary comparison does not by itself cover this setting: conformal factors pulled back to the disk can be unbounded near a corner. We therefore use smooth inner approximation, together with an approximation of the metric that preserves the curvature upper bound. Spectral convergence yields the inequality, but equality requires a separate argument that prevents the balanced poles from escaping to the boundary and detects strictness in the limiting radial comparison.}

Following the notation of \cite{LangfordLaugesen2023}, let $\Omega\subset\mathbb C$ be a bounded simply connected Lipschitz domain equipped with the metric
\[
  g=\omega(z)|dz|^2,
\]
where $\omega$ is a weight. If $\omega\in C^2(\Omega)$, then its curvature is
\[
  K_g=-\frac{\Delta\log\omega}{2\omega},
\]
where $\Delta$ continues to denote the classical Euclidean Laplacian. In \cite{LangfordLaugesen2023}, one assumes $\omega\in C^2(\Omega)\cap C(\overline\Omega)$ and $\omega>0$ on $\overline\Omega$. We relax that hypothesis to $\omega\in C^2(\Om)\cap L^\infty(\Om)$ with positivity required only in $\Om$.

Below, $\mu_k(\Om;\omega)$ denotes the $k$th Neumann eigenvalue for the conformal metric $g=\omega|dz|^2$.

\begin{theorem}\label{thm:lipschitz-main}
Let $\Om\subset\mathbb C$ be a bounded simply connected Lipschitz domain, and let $\omega\in C^2(\Om)\cap L^\infty(\Om)$ satisfy $\omega>0$ in $\Om$. If
\[
  -\frac{\Delta\log\omega}{2\omega}\leq K,
  \qquad K>0,
  \qquad K|\Om|_\omega<4\pi,
\]
where $|\Om|_\omega=\int_\Om\omega\,dA$, then
\begin{equation}\label{eq:main-lipschitz}
  \frac1{\mu_2(\Om;\omega)}+\frac1{\mu_3(\Om;\omega)}
  \geq \frac2{\mu_2(D_K(|\Om|_\omega))}.
\end{equation}
Equality holds if and only if there exists a conformal diffeomorphism $\Phi:\D\to\Om$ such that
\begin{equation}\label{eq:lip-rigidity-model}
	\omega(\Phi(z))|\Phi'(z)|^2
	=\frac{M(1-b)}{\pi(1-b+b|z|^2)^2},
	\qquad M=|\Om|_\omega,\quad b=\frac{KM}{4\pi}.
\end{equation}
Here
    \[\frac{M(1-b)}{\pi(1-b+b|z|^2)^2}\]
is the metric coefficient obtained by mapping the spherical cap to the unit disk $\mathbb D$; see Subsection~\ref{subsec:lip-rigidity}. Equivalently, the interior of $(\Om,\omega|dz|^2)$ is isometric to the interior of $D_K(M)$.
\end{theorem}

\begin{corollary}\label{cor:single-eigenvalue}
{Under the assumptions of Theorem~\ref{thm:lipschitz-main}, write $\mu_j(\Om)=\mu_j(\Om;\omega)$ and $M=|\Om|_\omega$. Then}
\[
  \mu_2(\Om)\leq\mu_2(D_K(M)).
\]
{Equality holds if and only if the interior of $(\Om,\omega|dz|^2)$ is isometric to the interior of $D_K(M)$.}
\end{corollary}

Our higher-dimensional counterexample to Conjecture~1.3 is as follows.

\begin{theorem}\label{thm:highdim-counterexample}
For every fixed integer $n\ge3$, there exists a smooth simply connected domain $\Om\subset\mathbb S^n$ such that
\[
 \frac{|\mathbb S^n|}{2}<|\Om|<|\mathbb S^n|.
\]
If $B\subset\mathbb S^n$ is a geodesic ball with the same volume as $\Om$, then
\begin{equation}\label{eq:highdim-main}
 \mu_2(\Om)>\mu_2(B),\qquad
 \sum_{j=2}^{n+1}\frac1{\mu_j(\Om)}<\frac n{\mu_2(B)}.
\end{equation}
The domain $\Om$ may be chosen as a geodesic ball with $2n$ mutually disjoint interior oblate ellipsoidal holes removed; its boundary has $2n+1$ connected components.
\end{theorem}

{In contrast to the two-dimensional case, even within the class of smooth simply connected domains, a
geodesic ball of the same volume need not maximize the first positive
Neumann eigenvalue on a higher-dimensional sphere.  This construction
does not decide the weaker problem under the volume constraint of at
most one hemisphere. Whether that weaker version holds remains open.}

{Section~\ref{sec:radial} records the smooth harmonic-mean comparison and the concepts of Green ridal function used below. Section~\ref{sec:lipschitz} proves the Lipschitz extension and its equality rigidity, including the required domain and metric approximation. Section~\ref{sec:highdim-counterexample} constructs the higher-dimensional counterexamples.}

\section{{Smooth harmonic-mean comparison and Green radial function}}\label{sec:radial}
{
The smooth harmonic-mean comparison is following.
\begin{proposition}[{\cite[Corollary~3.2]{ChenYang2026}}]\label{thm:smooth-main}
Let $(\Om,g)$ be a compact simply connected smooth Riemannian surface with nonempty smooth boundary. If $K_g\leq K$, $K>0$, and $M=|\Om|_g<4\pi/K$, then
\begin{equation}\label{eq:main-smooth}
 \frac1{\mu_2(\Om,g)}+\frac1{\mu_3(\Om,g)}
 \geq\frac2{\mu_2(D_K(M))}.
\end{equation}
\end{proposition}

For the proof in the Lipschitz case, we also require some intermediate results for the smooth case., we also need the intermediate radial value in the smooth argument. Retain the assumptions of Proposition~\ref{thm:smooth-main}. For an interior point $p$, let $\psi_p$ be the Dirichlet Green function normalized by $-\Delta_g\psi_p=\delta_p$ and $\psi_p=0$ on $\partial\Om$. Choose an orientation-preserving conformal map $\Phi_p:\D\to\Om$ with $\Phi_p(0)=p$. With the Hodge-star convention $\alpha\wedge\star_g\beta=\langle\alpha,\beta\rangle_g\,dv_g$, set $A_p=-2\pi\star_g d\psi_p$. Then
\begin{equation}\label{eq:potential-pullback}
 \psi_p\circ\Phi_p(z)=-\frac{\log|z|}{2\pi},
 \qquad \Phi_p^*A_p=d\theta.
\end{equation}
The potential has integer flux one. Hence there is a single-valued phase $e^{i\Theta_p}$ with $d\Theta_p=A_p$ locally and, after normalization, $\Phi_p^*e^{i\Theta_p}=e^{i\theta}$.

Write $d^{A_p}u=du-iA_pu$. The magnetic form domain $H^1_{A_p}(\Om;\mathbb C)$ is the completion of smooth functions on $\overline\Om$ that vanish near $p$, for the norm squared $\int_\Om(|u|^2+|d^{A_p}u|^2)\,dv_g$. The integer-flux gauge isomorphism gives
\begin{equation}\label{eq:gauge-isometry}
 \begin{gathered}
 U_p:H^1(\Om;\mathbb C)\longrightarrow H^1_{A_p}(\Om;\mathbb C),
 \qquad U_pf=e^{i\Theta_p}f,\\
 \int_\Om|d^{A_p}U_pf|^2\,dv_g=\int_\Om|df|^2\,dv_g,
 \qquad \int_\Om|U_pf|^2\,dv_g=\int_\Om|f|^2\,dv_g,
 \end{gathered}
\end{equation}
and
\begin{equation}\label{eq:U_p-inverse}
 U_p^{-1}u=e^{-i\Theta_p}u\in H^1(\Om;\mathbb C).
\end{equation}
See \cite[Section~2.3 and Lemma~7.3]{ProvenzanoSavo2026} and \cite[Appendix~A.1]{ColboisProvenzanoSavo2022} for these identities and the form-domain convention.

Let $\mathcal R_p$ be the real subspace of $H^1_{A_p}$ consisting of functions whose pullbacks by $\Phi_p$ depend only on $|z|$. Define
\begin{equation}\label{eq:kappa-def}
 \kappa_p=\inf_{0\ne u\in\mathcal R_p}
 \frac{\displaystyle\int_\Om|d^{A_p}u|^2\,dv_g}
      {\displaystyle\int_\Om u^2\,dv_g}.
\end{equation}
Its positive $L^2$-normalized minimizer is denoted by $u_p$; see \cite[Lemma~3.1 and the proof of Lemma~5.1]{ProvenzanoSavo2026}. This is the first eigenvalue of the restricted Green-radial problem, not of the full magnetic operator.

\begin{proposition}\label{prop:balanced}
Under the assumptions of Proposition~\ref{thm:smooth-main}, there is a pole $\bar p\in\Om$ such that
\[
 \int_\Om u_{\bar p}e^{-i\Theta_{\bar p}}\,dv_g=0
\]
and
\begin{equation}\label{eq:balanced-bound}
 \frac2{\mu_2(\Om,g)^{-1}+\mu_3(\Om,g)^{-1}}
 \leq\kappa_{\bar p}\leq\mu_2(D_K(M)).
\end{equation}
\end{proposition}
The centering and radial comparison are those of \cite[Theorems~4.1--4.3 and Claim~7.1]{ProvenzanoSavo2026}. The first inequality is the intermediate estimate in \cite[Section~3, formula~(3.5)]{ChenYang2026}, obtained by applying the reciprocal variational principle to the two real phase components. As in their Corollary~3.2, this argument uses the curvature upper bound through radial comparison and applies to the surfaces considered here. We shall use \eqref{eq:balanced-bound} on smooth approximations without repeating that variational proof.
}

\section{The harmonic-mean inequality and equality rigidity in the Lipschitz setting}\label{sec:lipschitz}
This section proves Theorem~\ref{thm:lipschitz-main}. Recall the variational characterization of the Neumann eigenvalues:
\begin{equation}\label{eq:weighted-minmax}
  \mu_k(\Om;\omega):=\min_{\substack{V\subset H^1(\Om)\\\dim V=k}}
   \max_{0\neq u\in V}
   \frac{\displaystyle\int_\Om|\nabla u|^2\,dA}
        {\displaystyle\int_\Om u^2\omega\,dA};
\end{equation}
see \cite[Section~4.5]{deb1995}.
{The possible vanishing of $\omega$ at the boundary does not obstruct this variational definition. Fix a disk $U_0\Subset\Om$. The Poincar\'e inequality on the connected Lipschitz domain $\Om$, together with $\min_{\overline{U_0}}\omega>0$, gives
\[
 \|u\|_{H^1(\Om)}^2\leq C\left(\int_\Om|\nabla u|^2\,dA+\int_\Om u^2\omega\,dA\right).
\]
Thus the form norm is equivalent to the usual $H^1$ norm. Moreover, $\omega\in L^\infty(\Om)$ and Rellich compactness imply that $H^1(\Om)\hookrightarrow L^2(\Om,\omega dA)$ is compact. The Neumann form therefore has discrete spectrum and the stated min--max characterization.}
The following lemma is the basis of the smooth-approximation argument.

\begin{lemma}\label{lem:spectral-convergence}
Let $\Om\subset\mathbb R^2$ be a bounded connected Lipschitz domain, and let $\omega\in C(\Om)\cap L^\infty(\Om)$ be positive in $\Om$. There exist nested smooth domains $\Om_j\Subset\Om_{j+1}\Subset\Om$, having the same number of boundary components as $\Om$, such that
    \begin{equation}\label{eq:spectral-convergence}
    	\bigcup_j\Om_j=\Om,
    	\qquad
    	\mu_k(\Om_j;\omega)\longrightarrow\mu_k(\Om;\omega)
    	\quad\text{for every fixed }k.
    \end{equation}
If $\Om$ is simply connected, the domains $\Om_j$ may be chosen simply connected.
\end{lemma}

\begin{proof}
We first construct the smooth approximation. Let $d(x)$ be the signed distance to $\partial\Om$, positive in $\Om$. By \cite[Proposition~3.1]{BallZarnescu2017}, there is a regularized distance function
\[
  \rho\in C^\infty(\mathbb R^2\setminus\partial\Om)
        \cap C^{0,1}(\mathbb R^2),
  \qquad \frac12\leq\frac{\rho(x)}{d(x)}\leq2
  \quad(x\notin\partial\Om),
\]
such that $\nabla\rho$ does not vanish at points near, but not on, the boundary. Thus, for all sufficiently small $\varepsilon>0$, the domains enclosed by the level sets,
\[
  \Om_\varepsilon=\{x:\rho(x)>\varepsilon\}
\]
have smooth boundary. Moreover, \cite[Theorem~5.1(i)--(ii)]{BallZarnescu2017} guarantees that these domains form a nested exhaustion of $\Om$ from within and that there is a homeomorphism carrying $\overline\Om$ onto $\overline{\Om_\varepsilon}$ and $\partial\Om$ onto $\partial\Om_\varepsilon$. Choosing sufficiently small $\varepsilon_j\downarrow0$ and setting $\Om_j=\Om_{\varepsilon_j}$ therefore yields the required nested inner approximation. The number of boundary components is preserved, and if $\Om$ is simply connected then every $\Om_j$ is simply connected as well.

We next establish uniformity in $j$ of the boundary estimates. Since $\Om$ is a Lipschitz domain, finitely many boundary coordinate neighborhoods can be fixed so that in each one $\Om$ lies above a Lipschitz graph. By \cite[Remark~5.3]{BallZarnescu2017}, in the portion of $\Om$ contained in each coordinate neighborhood there are constants $c,C>0$, independent of $\varepsilon$, such that
\[
  \partial_{x_2}\rho\geq c>0,
  \qquad |\nabla\rho|\leq C
\]
where $c,C$ may depend on the coordinate neighborhood. After shrinking the coordinate neighborhoods and decreasing $\varepsilon_1$, every $\partial\Om_j$ can be written in those same neighborhoods as a smooth graph $x_2=\varphi_j(x_1)$. From
\[
  \rho(x_1,\varphi_j(x_1))=\varepsilon_j
  \quad\Longrightarrow\quad
  \varphi_j'(x_1)
  =-\frac{\partial_{x_1}\rho}{\partial_{x_2}\rho}
       (x_1,\varphi_j(x_1)),
  \qquad |\varphi_j'|\leq C/c,
\]
and the finiteness of the collection of coordinate neighborhoods, we obtain a common slope bound $L$ controlling the Lipschitz graph at every point in every neighborhood.

By \cite[Proposition~III.2]{Chenais1975}, the domains $\Omega_j$ satisfy a uniform cone property. More precisely, for every $q\in\partial\Om_j$ there is an inward-pointing cone $\mathcal C_{j,q}$ with vertex at the origin, aperture parameter $\theta$, and height $h$, such that
\[
  \bigl(B(q,r)\cap\Om_j\bigr)+\mathcal C_{j,q}
  \subset\Om_j.
\]
The direction of the cone may depend on $q$ and $j$, while $\theta,h,r$ remain fixed. Choose a ball $B$ containing $\overline\Om$. Applying \cite[Theorem~II.1]{Chenais1975} with $m=1$, first extending functions to $\mathbb R^2$ and then restricting them to $B$, gives linear extension operators
    \begin{equation}\label{eq:extesion}
    	\mathcal E_j:H^1(\Om_j)\longrightarrow H^1(B),
    	\qquad (\mathcal E_jv)|_{\Om_j}=v,
    	\qquad
    	\|\mathcal E_jv\|_{H^1(B)}
    	\leq C_E\|v\|_{H^1(\Om_j)},
    \end{equation}
where $C_E$ is independent of $j$.

Fix an open disk $U_0$ such that $\overline{U_0}\subset\Om_1$. The uniform extension property above and connectedness of the domains give the following Poincar\'{e} estimate:
\begin{equation}\label{eq:anchored-poincare}
  \|v\|_{L^2(\Om_j)}^2
  \leq C\left(\|\nabla v\|_{L^2(\Om_j)}^2
        +\|v\|_{L^2(U_0)}^2\right),
  \qquad v\in H^1(\Om_j),
\end{equation}
where $C$ is independent of $j$. Here are the details.

First we prove \eqref{eq:anchored-poincare} with a constant $C_j$ that may depend on $j$. Let
\[
  m_j=\frac1{|\Om_j|}\int_{\Om_j}v\,dA{,}
\]
be the average of $v$ over $\Om_j$. The usual Poincar\'{e} inequality gives a constant $P_j$, possibly depending on $j$, such that
\[
  \|v-m_j\|_{L^2(\Om_j)}^2
  \leq P_j\|\nabla v\|_{L^2(\Om_j)}^2.
\]
For $m_j$, we have
\[
  \begin{aligned}
    |U_0|\,|m_j|^2
    &=\int_{U_0}|v-(v-m_j)|^2\,dA\\
    &\leq2\|v\|_{L^2(U_0)}^2
          +2\|v-m_j\|_{L^2(U_0)}^2\\
    &\leq2\|v\|_{L^2(U_0)}^2
          +2P_j\|\nabla v\|_{L^2(\Om_j)}^2.
  \end{aligned}
\]
Since $\int_{\Om_j}(v-m_j)\,dA=0$, we also have the orthogonal decomposition
\[
  \|v\|_{L^2(\Om_j)}^2
  =\|v-m_j\|_{L^2(\Om_j)}^2+|\Om_j|\,|m_j|^2.
\]
Substituting the preceding two estimates gives
    \begin{equation}\label{eq:local-estimate}
    	\|v\|_{L^2(\Om_j)}^2
    	\leq P_j\left(1+\frac{2|\Om_j|}{|U_0|}\right)
    	\|\nabla v\|_{L^2(\Om_j)}^2+\frac{2|\Om_j|}{|U_0|}\|v\|_{L^2(U_0)}^2.
    \end{equation}

We now show that the constant in \eqref{eq:anchored-poincare} can be chosen independently of $j$. Otherwise, there would be $j_n\to\infty$ and $v_n\in H^1(\Om_{j_n})$ which, after normalization, satisfy
    \begin{equation}\label{eq:v_n}
    	\|v_n\|_{L^2(\Om_{j_n})}=1,
    	\qquad
    	\|\nabla v_n\|_{L^2(\Om_{j_n})}^2
    	+\|v_n\|_{L^2(U_0)}^2\longrightarrow0.
    \end{equation}
Uniform extension and Rellich compactness yield a subsequence such that $\mathcal E_{j_n}v_n\to v$ in $L^2(B)$ and weakly in $H^1(B)$. The second limit in \eqref{eq:v_n} implies that $\nabla v=0$ in $\Omega$ and $v=0$ in $U_0$. By connectedness, $v=0$ throughout $\Om$; but
\[
  1=\|\mathcal E_{j_n}v_n\|_{L^2(\Om_{j_n})}
  \leq\|\mathcal E_{j_n}v_n-v\|_{L^2(\Om)}
  \longrightarrow0,
\]
a contradiction. This proves that the constant in \eqref{eq:anchored-poincare} is uniform.

Moreover, by the continuity and positivity of $\omega$ in $\Om$,
\[
  c_0:=\min_{\overline{U_0}}\omega>0,
  \qquad
  \|v\|_{L^2(U_0)}^2
  \leq\frac1{c_0}\int_{U_0}|v|^2\omega\,dA
  \leq\frac1{c_0}\|v\|_{L^2(\Om_j,\omega dA)}^2.
\]
Substituting this into \eqref{eq:anchored-poincare} gives
\begin{equation}\label{eq:weighted-H1-control}
  \|v\|_{H^1(\Om_j)}^2
  \leq C'\left(\|\nabla v\|_{L^2(\Om_j)}^2
        +\|v\|_{L^2(\Om_j,\omega dA)}^2\right),
\end{equation}
again uniformly in $j$.

Finally, we prove spectral convergence. First choose the first $k$ eigenfunctions $f_1,\ldots,f_k$ on $\Om$, orthonormal in weighted $L^2$, so that
\[
  \int_\Om f_\ell f_m\omega\,dA=\delta_{\ell m}.
\]
For $c=(c_1,\ldots,c_k)\in\mathbb R^k$, write $f_c=\sum_{\ell=1}^k c_\ell f_\ell$ and $|c|^2=\sum_{\ell=1}^k c_\ell^2$. Then
\[
  \int_{\Om_j}|\nabla f_c|^2\,dA
  \leq\int_\Om|\nabla f_c|^2\,dA
  =\sum_{\ell=1}^k\mu_\ell(\Om;\omega)c_\ell^2
  \leq\mu_k(\Om;\omega)|c|^2.
\]
Let
\[
  \eta_j
  :=\sum_{\ell=1}^k\int_{\Om\setminus\Om_j}f_\ell^2\omega\,dA
  \longrightarrow0;
\]
By the Cauchy--Schwarz inequality $|f_c|^2\leq|c|^2\sum_{\ell=1}^k f_\ell^2$, simultaneously for all $c$ we have
\[
  \int_{\Om_j}f_c^2\omega\,dA
  =|c|^2-\int_{\Om\setminus\Om_j}f_c^2\omega\,dA
  \geq(1-\eta_j)|c|^2.
\]
For all sufficiently large $j$, $\eta_j<1$, so $f_c|_{\Om_j}$ is nonzero whenever $c\ne0$. Thus the restricted functions remain linearly independent, and
$V_j=\Span\{f_1|_{\Om_j},\ldots,f_k|_{\Om_j}\}\subset H^1(\Om_j)$ is a $k$-dimensional trial space. Applying the preceding numerator and denominator estimates in \eqref{eq:weighted-minmax} gives
\[
  \mu_k(\Om_j;\omega)
  \leq\max_{0\neq c\in\mathbb R^k}
       \frac{\displaystyle\int_{\Om_j}|\nabla f_c|^2\,dA}
            {\displaystyle\int_{\Om_j}f_c^2\omega\,dA}
  \leq\frac{\mu_k(\Om;\omega)}{1-\eta_j}.
\]
Letting $j\to\infty$, we obtain
    \begin{equation}\label{eq:mu_k-upper-bound}
    	\limsup_{j\to\infty}\mu_k(\Om_j;\omega)
    	\leq\mu_k(\Om;\omega).
    \end{equation}
Conversely, let $u_{j,1},\ldots,u_{j,k}$ be weighted-$L^2$ orthonormal eigenfunctions on $\Om_j$. The uniform extensions \eqref{eq:extesion}, estimate \eqref{eq:weighted-H1-control}, and upper bound \eqref{eq:mu_k-upper-bound} show that $\mathcal E_ju_{j,\ell}$ is bounded in $H^1(B)$. Along a subsequence attaining the lower limit, these functions converge weakly in $H^1(B)$ and strongly in $L^2(B)$ to functions $u_1,\ldots,u_k$. We need strong convergence in weighted $L^2$. Extend $u_{j,\ell}$ by zero on $\Om\setminus\Om_j$ and denote the result by $\widehat u_{j,\ell}$. Here this extension is used only as a weighted-$L^2$ function; we do not take the gradient of the zero extension. Then
\[
  \begin{aligned}
  \|\widehat u_{j,\ell}-u_\ell\|_{L^2(\Om,\omega dA)}^2
  &=
  \int_{\Om_j}|\mathcal E_ju_{j,\ell}-u_\ell|^2\omega\,dA
  +\int_{\Om\setminus\Om_j}|u_\ell|^2\omega\,dA\\
  &\leq\|\omega\|_{L^\infty(\Om)}
      \|\mathcal E_ju_{j,\ell}-u_\ell\|_{L^2(\Om)}^2
  +\int_{\Om\setminus\Om_j}|u_\ell|^2\omega\,dA
  \longrightarrow0.
  \end{aligned}
\]
The last term tends to zero by absolute continuity of the integral. Strong convergence in weighted $L^2$ therefore gives
\[
  \int_\Om u_\ell u_m\omega\,dA=\delta_{\ell m};
\]
that is, $u_1,\ldots,u_k$ are orthonormal in weighted $L^2$. For every $c=(c_1,\ldots,c_k)$ and every $U\Subset\Om$, we have $U\subset\Om_j$ for all sufficiently large $j$. Weak lower semicontinuity on this fixed set, where $\mathcal E_ju_{j,\ell}=u_{j,\ell}$, gives
\[
  \int_U\left|\nabla\sum_{\ell=1}^k c_\ell u_\ell\right|^2dA
  \leq
  \left(\liminf_{j\to\infty}\mu_k(\Om_j;\omega)\right)|c|^2.
\]
Since $U$ is arbitrary, the same estimate holds with $\Om$ in place of $U$. The span of $u_1,\ldots,u_k$ is $k$-dimensional, so \eqref{eq:weighted-minmax} gives
\[
  \mu_k(\Om;\omega)
  \leq\liminf_{j\to\infty}\mu_k(\Om_j;\omega).
\]
Together with the upper bound \eqref{eq:mu_k-upper-bound}, this proves spectral convergence.
\end{proof}

{
\begin{lemma}\label{lem:smooth-metric-approximation}
Under the assumptions of Theorem~\ref{thm:lipschitz-main}, choose the smooth inner exhaustion from Lemma~\ref{lem:spectral-convergence}. There are positive weights $\omega_j\in C^\infty(\overline\Om_j)$ and numbers $\varepsilon_j\downarrow0$ such that
\begin{equation}\label{eq:smooth-metric-approximation}
 \left\|\frac{\omega_j}{\omega}-1\right\|_{L^\infty(\Om_j)}\leq\varepsilon_j,
 \qquad \omega_j\longrightarrow\omega\quad\hbox{in }C^2_{\mathrm{loc}}(\Om),
 \qquad -\frac{\Delta\log\omega_j}{2\omega_j}\leq K.
\end{equation}
With $g_j=\omega_j|dz|^2$ and $M_j=\int_{\Om_j}\omega_j\,dA$, one has $M_j<M=|\Om|_\omega$, $M_j\to M$, and
\begin{equation}\label{eq:smooth-metric-spectral-convergence}
 \mu_k(\Om_j;\omega_j)\longrightarrow\mu_k(\Om;\omega)
 \quad\hbox{for every fixed }k.
\end{equation}
In particular, $KM_j<4\pi$ for every $j$.
\end{lemma}
\begin{proof}
For each fixed $j$, the weight $\omega$ is $C^2$ and has a positive lower bound on a neighborhood of $\overline\Om_j$. Mollification in that neighborhood gives positive smooth weights $\widetilde\omega_{j,t}\to\omega$ in $C^2(\overline\Om_j)$ as $t\downarrow0$. Put
\[
 \delta_{j,t}=\max\left\{0,\sup_{\overline\Om_j}
 \left(-\frac{\Delta\log\widetilde\omega_{j,t}}{2\widetilde\omega_{j,t}}-K\right)\right\}.
\]
Then $\delta_{j,t}\to0$. Let
\[
 \widehat\omega_{j,t}=(1+\delta_{j,t}/K)\widetilde\omega_{j,t}.
\]
Then $\widehat\omega_{j,t}$ has curvature upper bound $K$ and still converges to $\omega$ in $C^2(\overline\Om_j)$. Fix $\varepsilon_j=1/(j+2)$ and set $m_j=\int_{\Om_j}\omega\,dA<M$. Choose $t_j$ sufficiently small and put $\omega_j=\widehat\omega_{j,t_j}$ so that the relative error in \eqref{eq:smooth-metric-approximation} holds, $\|\omega_j-\omega\|_{C^2(\overline\Om_j)}\leq1/j$, and $|M_j-m_j|<(M-m_j)/2$. Thus $M_j<M$, local $C^2$ convergence holds, and the relative error gives $M_j\to M$. The min--max bounds are
\[
 \frac{\mu_k(\Om_j;\omega)}{1+\varepsilon_j}
 \leq\mu_k(\Om_j;\omega_j)
 \leq\frac{\mu_k(\Om_j;\omega)}{1-\varepsilon_j}.
\]
Lemma~\ref{lem:spectral-convergence} proves \eqref{eq:smooth-metric-spectral-convergence}. Finally, $M_j<M$ implies $KM_j<4\pi$.
\end{proof}
}

{
\begin{proof}[Proof of the inequality in Theorem~\ref{thm:lipschitz-main}]
Take $(\Om_j,g_j)$ and $M_j$ from Lemma~\ref{lem:smooth-metric-approximation}, with $KM_j<4\pi$. These surfaces have smooth boundary and smooth positive metrics, so Proposition~\ref{thm:smooth-main} applies and gives
\begin{equation}\label{eq:harmonic-average-smooth-approximation}
 \frac1{\mu_2(\Om_j;\omega_j)}+\frac1{\mu_3(\Om_j;\omega_j)}
 \geq\frac2{\mu_2(D_K(M_j))}.
\end{equation}
The model eigenvalue is continuous in area by \cite[Theorem~1(i)]{LangfordLaugesenScaling2023}. Passing to the limit using \eqref{eq:smooth-metric-spectral-convergence} proves \eqref{eq:main-lipschitz}.
\end{proof}
}

\subsection{Equality rigidity}\label{subsec:lip-rigidity}
We now prove the equality rigidity statement in \eqref{eq:main-lipschitz}. We begin by mapping the spherical cap to the unit disk $\mathbb D$.

\begin{lemma}\label{lem:lip-rigidity-model-metric}
Let $K>0$ and $0<M<4\pi/K$. Then the interior of the spherical cap $D_K(M)$ of area $M$ is isometric to the unit disk $\D$ equipped with the conformal metric
\begin{equation}\label{eq:lip-rigidity-model-density}
 p_{K,M}(r)|dz|^2,
 \qquad p_{K,M}(r)=\frac{M(1-b)}{\pi(1-b+br^2)^2},
 \qquad b=\frac{KM}{4\pi}\in(0,1).
\end{equation}
Here $r=|z|$ is the Euclidean radius in the unit disk, and $z=0$ corresponds to the center of the spherical cap.
\end{lemma}
\begin{proof}
The sphere of curvature $K$ has radius $R=K^{-1/2}$. Center the spherical cap at the north pole, and let $\theta$ and $\varphi$ denote the polar and azimuthal angles. The spherical metric is
\[
 g_K=\frac1K\bigl(\dd\theta^2+\sin^2\theta\,\dd\varphi^2\bigr),
\]
and the model cap is given by $0\le\theta<\Theta$, where $0<\Theta<\pi$.

We first derive the coordinate formula for stereographic projection. After dividing the ambient coordinates by $R$, a point on the sphere is written as
\[
 P=(\sin\theta\cos\varphi,\,\sin\theta\sin\varphi,\,\cos\theta).
\]
Project stereographically from the south pole $S=(0,0,-1)$ onto the equatorial plane; that is, let $Q$ be the intersection of the line $S+\lambda(P-S)$ with the plane $X_3=0$. The third coordinate gives
\[
 -1+\lambda(1+\cos\theta)=0,
 \qquad \lambda=\frac1{1+\cos\theta}.
\]
Thus the complex coordinate of the intersection point is
\[
 w=Q_1+iQ_2
   =\frac{\sin\theta}{1+\cos\theta}e^{i\varphi}
   =\tan\frac\theta2\,e^{i\varphi}.
\]
Hence, after setting $t=\tan(\theta/2)$, stereographic projection is exactly $w=te^{i\varphi}$; it maps the north pole to the origin of the plane.

Using $\dd\theta=2\dd t/(1+t^2)$ and $\sin\theta=2t/(1+t^2)$, we obtain
\[
 g_K=\frac{4}{K(1+t^2)^2}
       \bigl(\dd t^2+t^2\dd\varphi^2\bigr)
     =\frac{4}{K(1+|w|^2)^2}|dw|^2.
\]
Let $a=\tan(\Theta/2)$. The image of the cap is the planar disk $|w|<a$. Rescaling by $w=az$, and hence pulling the metric back to the unit disk, gives
\[
 g_K=\frac{4a^2}{K(1+a^2r^2)^2}|dz|^2,
 \qquad r=|z|<1.
\]
Finally, the area condition for the spherical cap gives
\[
 M=\frac{2\pi}{K}(1-\cos\Theta)=\frac{4\pi}{K}\frac{a^2}{1+a^2}.
\]
Thus $b=a^2/(1+a^2)$, or $a^2=b/(1-b)$. Substitution into the metric coefficient above yields
\[
 \frac{4a^2}{K(1+a^2r^2)^2}
 =\frac{4b(1-b)}{K(1-b+br^2)^2}
 =\frac{M(1-b)}{\pi(1-b+br^2)^2}.
\]
This is \eqref{eq:lip-rigidity-model-density}. The composition of inverse stereographic projection with the dilation $z\mapsto az$ gives the required isometry, and $z=0$ corresponds to the center of the cap.
\end{proof}

Below we write $\lambda_K(a)=\mu_2(D_K(a))$.

\begin{lemma}\label{lem:lip-rigidity-model-gap}
If $K>0$ and $KM<4\pi$, then
\begin{equation}\label{eq:lip-rigidity-model-gap}
 \lambda_K(M)>\frac{2\pi}{M}.
\end{equation}
\end{lemma}
\begin{proof}
The first positive Neumann eigenvalue of the model disk has multiplicity two, and its eigenspace is spanned by $v(r)\cos\theta$ and $v(r)\sin\theta$; see \cite[Section~6]{ProvenzanoSavo2026}.
In the coordinates of \eqref{eq:lip-rigidity-model-density}, let
\[
 E=2\pi\int_0^1\left(rv'(r)^2+\frac{v(r)^2}{r}\right)\dd r,
 \qquad N=\int_\D v(|z|)^2p_{K,M}(|z|)\dd A.
\]
We show that $\lambda_K(M)=E/N$. Write $p(r)=p_{K,M}(r)$ and take the eigenfunction $u(r,\theta)=v(r)\cos\theta$. For the two-dimensional conformal metric $g=p(r)|dz|^2$,
\[
 dv_g=p(r)\dd A,
 \qquad \int_\D|\nabla_g u|_g^2\,dv_g
 =\int_\D|\nabla u|^2\dd A.\]
Therefore,
\[
 \lambda_K(M)=\frac{\int_\D|\nabla u|^2\dd A}{\int_\D u^2p(r)\dd A}.
\]

In polar coordinates,
\[
 |\nabla u|^2
 =|\partial_r u|^2+\frac1{r^2}|\partial_\theta u|^2
 =v'(r)^2\cos^2\theta+\frac{v(r)^2}{r^2}\sin^2\theta.
\]
Since $\dd A=r\dd r\dd\theta$ and
\[
 \int_0^{2\pi}\cos^2\theta\dd\theta
 =\int_0^{2\pi}\sin^2\theta\dd\theta=\pi,
\]
the numerator is
\[
 \begin{aligned}
 \int_\D|\nabla u|^2\dd A
 &=\int_0^1\int_0^{2\pi}
   \left(v'(r)^2\cos^2\theta+
         \frac{v(r)^2}{r^2}\sin^2\theta\right)r\dd\theta\dd r\\
 &=\pi\int_0^1\left(rv'(r)^2+\frac{v(r)^2}{r}\right)\dd r
 =\frac E2.
 \end{aligned}
\]
whereas the denominator is
\[
 \int_\D u^2p(r)\dd A
 =\int_0^1\int_0^{2\pi}v(r)^2\cos^2\theta\,p(r)r\dd\theta\dd r
 =\pi\int_0^1v(r)^2p(r)r\dd r
 =\frac N2.
\]
Hence
    \begin{equation}\label{eq:lambda_KM}
    	\lambda_K(M)
    	=\frac{\displaystyle\int_\D|\nabla u|^2\dd A}
    	{\displaystyle\int_\D u^2p(r)\dd A}=\frac EN.
    \end{equation}

Let $s=\log r$ and $V(s)=v(e^s)$. Since $E<\infty$, we have $V\in H^1(({-\infty},0))$. Continuity of the eigenfunction $u(r,\theta)$ at the origin also gives $v(0)=0$. Therefore, for every $t\le0$,
\[
 |V(t)|^2=2\int_{-\infty}^t VV'\dd s
 \le\int_{-\infty}^0(V^2+V'^2)\dd s=\frac E{2\pi}.
\]
Thus $E\geq2\pi\lVert v\rVert_\infty^2$. On the other hand, $N<M\lVert v\rVert_\infty^2$. Combining these facts with \eqref{eq:lambda_KM} proves the claim.
\end{proof}

\begin{lemma}\label{lem:lip-rigidity-coefficient}
Fix $L>0$. Let $G_0,G_1$ be positive $C^1$ functions on $(0,L]$ such that $G_i(a)/a$ extends continuously to $a=0$, with value $4\pi$ there. Define
\begin{equation}\label{eq:lip-rigidity-coefficient-quotient}
 \kappa(G_i;L):=\inf_{0\ne u\in\mathcal F_L}
 \frac{\displaystyle\int_0^L
       \left(G_i u'^2+\frac{4\pi^2}{G_i}u^2\right)\dd a}
      {\displaystyle\int_0^Lu^2\dd a},
\end{equation}
where
\[
 \mathcal F_L=\left\{u\in H^1_{\mathrm{loc}}(0,L]:
  \int_0^L\left(au'^2+\frac{u^2}{a}\right)\dd a<\infty\right\}.
\]
If $G_1\ge G_0$, then $\kappa(G_1;L)\le\kappa(G_0;L)$; if the inequality is strict at some interior point, then the eigenvalue inequality is strict. Moreover, if a fixed function $G$ satisfies the preceding assumptions for $L\in(0,L_*]$, then $L\mapsto\kappa(G;L)$ is continuous on $(0,L_*)$.
\end{lemma}
\begin{proof}
{Monotonicity and strictness follow from the interpolation proof of \cite[Lemma~5.1]{ProvenzanoSavo2026}; see also \cite[Lemma~3.3]{MichettiProvenzanoSavo2026} for two general coefficients. The latter radial argument needs only positive flux, so applies with $\nu=1$. Since $G_i(a)\sim4\pi a$, the form domains coincide; the argument uses only first derivatives of the coefficients and therefore applies to the present $C^1$ setting.}

We next consider continuity. Suppose that $L\to L_0\in(0,L_*)$. Setting $a=Lx$ and $U(x)=u(Lx)$ transforms the Rayleigh quotient into
    \[
         \frac{\displaystyle\int_0^1
            \left(\frac{G(Lx)}{L^2}U'(x)^2+
              \frac{4\pi^2}{G(Lx)}U(x)^2\right)\dd x}
      {\displaystyle\int_0^1U(x)^2\dd x},
    \]
and all these quotients have the same domain $\mathcal F_1$. Let $h(a)=G(a)/a$ and set $h(0)=4\pi$. Then $h$ is continuous and positive on $[0,L_*]$. Since
        \[\frac{G(Lx)}{L^2}=\frac{xh(Lx)}{L},\quad\frac{1}{G(Lx)}=\frac{1}{Lxh(Lx)},\]
{the ratios $\bigl(G(Lx)/L^2\bigr)/\bigl(G(L_0x)/L_0^2\bigr)$ and $G(L_0x)/G(Lx)$ converge uniformly to $1$ for $0<x\leq1$, with continuous extensions at $x=0$. This follows from the uniform continuity and positive lower bound of $h$ near $[0,L_0]$.} Hence, for every $\varepsilon>0$, whenever $|L-L_0|$ is sufficiently small,
\[
 (1-\varepsilon)\kappa(G;L_0)
 \le\kappa(G;L)\le(1+\varepsilon)\kappa(G;L_0),
\]
which proves the required continuity.
\end{proof}

\begin{lemma}\label{lem:lip-rigidity-model-identification}
Let $K>0$ and $0<M<4\pi/K$. Set
\begin{equation}\label{eq:lip-rigidity-model-G}
 G_K(a)=a(4\pi-Ka),\qquad 0<a<M.
\end{equation}
Then
\begin{equation}\label{eq:lip-rigidity-model-identification}
 \kappa(G_K;M)=\lambda_K(M).
\end{equation}
\end{lemma}
\begin{proof}
The cumulative area function of the spherical cap \eqref{eq:lip-rigidity-model-density} is
\[
 P(r):=P_{K,M}(r)
 =\int_{|z|<r}p_{K,M}(|z|)\dd A
 =\frac{Mr^2}{1-b+br^2}.
\]
It maps $(0,1)$ strictly increasingly onto $(0,M)$ and satisfies
\begin{equation}\label{eq:lip-rigidity-model-area-identity}
 P'(r)=2\pi p_{K,M}(r)r,
 \qquad G_K(P(r))=2\pi rP'(r).
\end{equation}
For any $f\in\mathcal F_M$, let
\[
 v(r)=f(P(r)).
\]
Making the change of variables $a=P(r)$ and using $f'(P(r))=v'(r)/P'(r)$ together with \eqref{eq:lip-rigidity-model-area-identity}, we obtain
\begin{align*}
 \int_0^M G_K(a)f'(a)^2\dd a=2\pi\int_0^1rv'(r)^2\dd r,
\end{align*}
\begin{equation}
 	\int_0^M\frac{4\pi^2}{G_K(a)}f(a)^2\dd a=2\pi\int_0^1\frac{v(r)^2}{r}\dd r\quad\text{and}\quad\int_0^Mf(a)^2\dd a=2\pi\int_0^1v(r)^2p_{K,M}(r)r\dd r.
\end{equation}
Therefore,
\begin{equation}\label{eq:lip-rigidity-model-kappa-radial}
 \kappa(G_K;M)
 =\inf_{v\ne0}
 \frac{\displaystyle\int_0^1
 \left(rv'(r)^2+\frac{v(r)^2}{r}\right)\dd r}
 {\displaystyle\int_0^1v(r)^2p_{K,M}(r)r\dd r},
\end{equation}
where the infimum is taken over radial functions for which the numerator is finite.

On the other hand, take $w(r,\theta)=v(r)\cos\theta$ on the spherical cap. Conformal invariance of the two-dimensional Dirichlet energy gives
\[
 \frac{\displaystyle\int_{\D}|\nabla w|^2\dd A}
      {\displaystyle\int_{\D}w^2p_{K,M}\dd A}
 =\frac{\displaystyle\int_0^1
 \left(rv'(r)^2+\frac{v(r)^2}{r}\right)\dd r}
 {\displaystyle\int_0^1v(r)^2p_{K,M}(r)r\dd r}.
\]
Separation of variables shows that the eigenfunctions corresponding to the first positive Neumann eigenvalue have angular frequency $1$; see \cite[Section~6]{ProvenzanoSavo2026}. Thus the right-hand side of \eqref{eq:lip-rigidity-model-kappa-radial} is exactly $\mu_2(D_K(M))=\lambda_K(M)$, which proves \eqref{eq:lip-rigidity-model-identification}.
\end{proof}

\begin{proof}[Equality case of Theorem~\ref{thm:lipschitz-main}]
{
\emph{Step 1: An extremizing sequence on smooth inner approximations.}
Assume equality in \eqref{eq:main-lipschitz}, and set $M=|\Om|_\omega$. Choose the domains, metrics $g_j=\omega_j|dz|^2$, and areas $M_j$ from Lemma~\ref{lem:smooth-metric-approximation}. Proposition~\ref{prop:balanced} supplies poles $p_j$ and corresponding radial values $\kappa_j$, defined by \eqref{eq:kappa-def} on $(\Om_j,g_j,p_j)$, such that
\[
 \frac2{\mu_2(\Om_j;\omega_j)^{-1}+\mu_3(\Om_j;\omega_j)^{-1}}
 \leq\kappa_j\leq\lambda_K(M_j).
\]
Both outer quantities tend to $\lambda_K(M)$ by \eqref{eq:smooth-metric-spectral-convergence}, equality on $\Om$, and continuity of the model eigenvalue. Hence
\begin{equation}\label{eq:lip-rigidity-kappa-limit}
 \kappa_j\longrightarrow\lambda_K(M).
\end{equation}
All metric quantities on $\Om_j$ below are computed with $g_j$; in particular, $dv_g$ in integrals over $\Om_j$ means $\omega_j\,dA$.
}

\emph{Step 2: The poles do not escape to the boundary.}
Fix $q\in\Om_1$. Choose conformal maps $F_j:\D\to\Om_j$ such that $F_j(0)=q$ and $F_j'(0)>0$, and write $\varphi_j=F_j^{-1}$. Then $F_j$ and $\varphi_j$ converge locally uniformly to conformal maps $F:\D\to\Om$ and $\varphi=F^{-1}$, respectively, and all interior complex derivatives converge as well; see \cite[Theorem~1.8]{pc}.

Let $\zeta_j=\varphi_j(p_j)$ and pass to a subsequence such that $\zeta_j\to\zeta\in\overline\D$. Suppose that $|\zeta|=1$, and define
\[
 b_{\zeta_j}(z)=\frac{z-\zeta_j}{1-\overline{\zeta_j}z},
 \qquad W_j=b_{\zeta_j}\circ\varphi_j.
\]
Then $W_j(p_j)=0$, and $|W_j|$ is Green-radial about $p_j$. Applying \eqref{eq:potential-pullback} to the conformal coordinates $W_j^{-1}:\D\to\Om_j$ centered at $p_j$ shows that on $\Om_j\setminus\{p_j\}$,
\[
 A_{p_j}=\dd\arg W_j:=\text{Im}\left(\frac{dW_{j}}{W_{j}}\right).
\]
Choose the phase $\Theta_{p_j}$ so that
\[
 e^{i\Theta_{p_j}}=\frac{W_j}{|W_j|}.
\]
By \eqref{eq:U_p-inverse},
\[
 U_{p_j}^{-1}(|W_j|)
 =e^{-i\Theta_{p_j}}|W_j|
 =\overline{W_j}
 \qquad\text{on $\Om_j\setminus\{p_j\}$}.
\]
Conformal invariance and the area formula also give $\int_{\Om_j}|\nabla W_j|^2\dd A=2\pi$, so $\overline{W_j}\in H^1(\Om_j;\mathbb C)$. Thus \eqref{eq:gauge-isometry} ensures that $|W_j|\in\mathcal R_{p_j}$. Conformal invariance gives its magnetic energy:
    \begin{align}
    	E_{j}:&=\int_{\Omega_{j}}\big\lvert d^{A_{p_{j}}}\lvert W_{j}\rvert\big\rvert_{g}^2\ dv_{g}=\int_{\Omega_{j}}\big\lvert d^{A_{p_{j}}}\lvert W_{j}\rvert\big\rvert_{\text{Euc}}^2\ dA\\
    	&=\int_{\Om_j}\left(\big|\nabla|W_j|\big|^2+|W_j|^2|\nabla\arg W_j|^2\right)\dd A\\
    	&=\int_{\Omega_{j}}\lvert\nabla W_{j}\rvert^2\ dA=\int_\D|\nabla z|^2\dd A\\
    	&=2\pi.
    \end{align}
For each fixed $x\in\Om$, one has $x\in\Om_j$ for all sufficiently large $j$, and $|b_{\zeta_j}(\varphi_j(x))|\to1$. {Dominated convergence, using $0\le|W_j|^2\le1$, $\omega\in L^1(\Om)$, and the uniform relative estimate $\omega_j/\omega\to1$ on $\Om_j$ from \eqref{eq:smooth-metric-approximation}, gives}
\[
 {N_j:=\int_{\Om_j}|W_j|^2\omega_j\dd A}\longrightarrow M,
 \qquad \limsup_j\kappa_j\le\lim_j\frac{E_{j}}{N_j}
 =\frac{2\pi}{M}.
\]
This contradicts \eqref{eq:lip-rigidity-kappa-limit} and Lemma~\ref{lem:lip-rigidity-model-gap}. Therefore $\zeta\in\D$ and $p_j\to p=F(\zeta)\in\Om$.

\emph{Step 3: Recentering.} Let
\[
 \Phi_j=F_j\circ b_{\zeta_j}^{-1},\qquad
 \Phi=F\circ b_\zeta^{-1},\qquad
 {\sigma_j=\omega_j(\Phi_j)|\Phi_j'|^2,}
 \qquad \sigma=\omega(\Phi)|\Phi'|^2.
\]
{By the local $C^2$ convergence in \eqref{eq:smooth-metric-approximation} and the convergence of the conformal maps and their interior derivatives, $\sigma_j\to\sigma$ in $C^1$ on every compact subset of $\D$.} Define the area functions
\[
 \mathcal A_j(r)=\int_{|z|<r}\sigma_j\dd A,
 \qquad \mathcal A(r)=\int_{|z|<r}\sigma\dd A,
\]
and, using area as the variable, define
\begin{equation}\label{eq:lip-rigidity-area-coefficient}
 \begin{aligned}
 G_j(a)&=2\pi r\mathcal A_j'(r)
       &&\text{when}\ a=\mathcal A_j(r),\\
 G_p(a)&=2\pi r\mathcal A'(r)
       &&\text{when}\ a=\mathcal A(r).
 \end{aligned}
\end{equation}
Then
\begin{equation}\label{eq:lip-rigidity-G-convergence}
 G_j\longrightarrow G_p
  \quad\text{uniformly on every compact subinterval of $(0,M)$.}
\end{equation}
We next show that
    \begin{equation}\label{eq:lip-rigidity-kappa-coefficient}
    	\kappa_j=\kappa(G_j;M_j).
    \end{equation}
where $\kappa(G_j;M_j)$ is defined by \eqref{eq:lip-rigidity-coefficient-quotient}. Given $U\in\mathcal R_{p_j}$, write it in the conformal coordinates $\Phi_j$ centered at $p_j$ as
\[
 \Phi_j^*U(z)=v(|z|).
\]
By $\Phi_j^*g=\sigma_j|dz|^2$, $\Phi_j^*A_{p_j}=d\theta$, and conformal invariance of the two-dimensional magnetic Dirichlet energy,
\begin{equation}\label{eq:lip-rigidity-radial-energy}
 \int_{\Om_j}|d^{A_{p_j}}U|^2\dv=\int_{\mathbb{D}}\lvert d^{d\theta}v\rvert_{\text{Euc}}^2\ dA=2\pi\int_0^1\left(rv'(r)^2+\frac{v(r)^2}{r}\right)\dd r.
\end{equation}
On the other hand, since
\[
 \mathcal A_j'(r)
 =\int_0^{2\pi}\sigma_j(re^{i\theta})r\dd\theta,
\]
we have
\begin{equation}\label{eq:lip-rigidity-radial-mass}
 \int_{\Om_j}U^2\dv=\int_{\mathbb{D}}v^2\sigma_j\ dA=\int_0^1v(r)^2\mathcal A_j'(r)\dd r.
\end{equation}

Now make the area-variable substitution
\[
 a=\mathcal A_j(r),\qquad
 v(r)=u(\mathcal A_j(r)),\qquad
 \dd a=\mathcal A_j'(r)\dd r.
\]
Since $G_j(a)=2\pi r\mathcal A_j'(r)$,
\[
 2\pi\int_0^1rv'(r)^2\dd r
 =\int_0^{M_j}G_j(a)u'(a)^2\dd a,
\]
and
\[
 2\pi\int_0^1\frac{v(r)^2}{r}\dd r
 =\int_0^{M_j}\frac{4\pi^2}{G_j(a)}u(a)^2\dd a.
\]
Similarly, \eqref{eq:lip-rigidity-radial-mass} gives
\[
 \int_{\Om_j}U^2\dv=\int_0^{M_j}u(a)^2\dd a.
\]
{Since $\sigma_j$ is positive and $C^1$ near $0$, we have $\mathcal A_j(r)=\pi\sigma_j(0)r^2+O(r^3)$ and $\mathcal A_j'(r)=2\pi\sigma_j(0)r+O(r^2)$. Hence $G_j(a)/a\to4\pi$ as $a\downarrow0$, as required in Lemma~\ref{lem:lip-rigidity-coefficient}.}
Thus the original Rayleigh quotient becomes exactly
\[
 \frac{\displaystyle\int_0^{M_j}
 \left(G_j u'^2+\frac{4\pi^2}{G_j}u^2\right)\dd a}
 {\displaystyle\int_0^{M_j}u^2\dd a}.
\]
Since $\mathcal A_j:(0,1)\to(0,M_j)$ is strictly increasing, the preceding change of variables gives a one-to-one correspondence between finite-energy Green-radial functions and $\mathcal F_{M_j}$. Taking the infimum of the Rayleigh quotient and using \eqref{eq:lip-rigidity-coefficient-quotient} yields $\kappa_j=\kappa(G_j;M_j)$.

Applying the isoperimetric inequality under an upper curvature bound \cite{cifea1980} to each interior disk gives
    \begin{equation}\label{eq:isopremetric-inequality}
    	\mathcal A(r)(4\pi-K\mathcal A(r))\le\ell(r)^2,
    \end{equation}
where
    \[\ell(r)=\int_0^{2\pi}\sqrt{\sigma(re^{i\theta})}\,r\dd\theta.\]
The Cauchy--Schwarz inequality then gives
    \begin{align}
    	\ell(r)^2&\leq\left(\int_{0}^{2\pi}\sigma(re^{i\theta})\ d\theta\right)\left(\int_{0}^{2\pi}r^2\ d\theta\right)=2\pi r\mathcal{A}'(r),
    \end{align}
Combining this with \eqref{eq:isopremetric-inequality}, we obtain
    \begin{equation}\label{eq:lip-rigidity-circle-comparison}
    	\mathcal A(r)(4\pi-K\mathcal A(r))
    	\le\ell(r)^2\le2\pi r\mathcal A'(r).
    \end{equation}
Therefore,
\begin{equation}\label{eq:lip-rigidity-G-comparison}
 G_j(a)\ge G_K(a),\qquad G_p(a)\ge G_K(a),
\end{equation}
where $G_K(a)=a(4\pi-Ka)$.

We next prove that
    \begin{equation}\label{eq:lip-rigidity-G-equality}
    	G_p(a)=G_K(a)\qquad(0<a<M).
    \end{equation}
If $G_p(a_0)>G_K(a_0)$ at some $a_0\in(0,M)$, then continuity and \eqref{eq:lip-rigidity-G-convergence} allow us to choose a nonzero $h\in C_c^\infty((0,M))$, $h\ge0$, such that for all sufficiently large $j$,
    \begin{equation}\label{eq:tilde-G}
    	\widetilde G:=G_K+h\le G_j\quad\text{on }(0,M_j].
    \end{equation}
Equations~\eqref{eq:lip-rigidity-kappa-coefficient} and \eqref{eq:tilde-G}, together with Lemmas~\ref{lem:lip-rigidity-coefficient} and \ref{lem:lip-rigidity-model-identification}, give
\[
 \limsup_j\kappa_j=\limsup_j\kappa(G_{j};M_{j})
 \le\lim_j\kappa(\widetilde G;M_j)
 =\kappa(\widetilde G;M)
 <\kappa(G_K;M)=\lambda_K(M),
\]
contradicting \eqref{eq:lip-rigidity-kappa-limit}. Hence \eqref{eq:lip-rigidity-G-equality} holds.

\emph{Step 4: Interior isometry.}
By \eqref{eq:lip-rigidity-G-equality}, both inequalities in \eqref{eq:lip-rigidity-circle-comparison} are equalities. The equality condition in the Cauchy--Schwarz inequality shows that $\sigma(re^{i\theta})$ is independent of $\theta$. Moreover,
\[
 2\pi r\mathcal A'(r)=\mathcal A(r)(4\pi-K\mathcal A(r)),
 \qquad\lim_{r\uparrow1}\mathcal A(r)=M,
\]
has the solution
\[
 \mathcal A(r)=\frac{Mr^2}{1-b+br^2},
 \qquad \sigma(r)=\frac{\mathcal A'(r)}{2\pi r}
 =p_{K,M}(r).
\]
Therefore,
\begin{equation}\label{eq:lip-rigidity-isometry}
\omega(\Phi(z))|\Phi'(z)|^2=p_{K,M}(|z|),
\end{equation}
that is, $\Phi$ maps the interior of the model disk isometrically onto $(\Om,\omega|dz|^2)$.

{Conversely, suppose that \eqref{eq:lip-rigidity-isometry} holds. Composition with $\Phi$ preserves weighted $L^2$ mass and Dirichlet energy. It also identifies the form domains: if $f\in H^1(\D)$, then $u=f\circ\Phi^{-1}$ is locally $H^1$ on $\Om$ with finite Dirichlet energy and weighted $L^2$ mass. Applying \eqref{eq:weighted-H1-control} to $u|_{\Om_j}$ and letting $j\to\infty$ gives $u\in H^1(\Om)$. In the other direction, $p_{K,M}$ is bounded above and below by positive constants on $\overline\D$, so finite transformed mass and energy give $u\circ\Phi\in H^1(\D)$. The min--max characterization therefore identifies the Neumann spectra. In particular, $\mu_2(\Om;\omega)=\mu_3(\Om;\omega)=\lambda_K(M)$, and equality follows.}
\end{proof}

\section{Higher-dimensional counterexamples to Conjecture 1.3}
\label{sec:highdim-counterexample}

\begingroup
\theoremstyle{plain}
\newtheorem{chapterSixLemma}[theorem]{Lemma}
\theoremstyle{remark}
\newtheorem{chapterSixRemark}[theorem]{Remark}
\let\lemma\chapterSixLemma
\let\endlemma\endchapterSixLemma
\let\remark\chapterSixRemark
\let\endremark\endchapterSixRemark
\renewcommand{\proofname}{Proof}

{In this section we prove Theorem~\ref{thm:highdim-counterexample}. We remove $2n$ small oblate ellipsoidal holes from a sufficiently large spherical cap and compare the resulting domain with a smaller cap of the same volume. The comparison uses positivity of the derivative of the cap eigenvalue with respect to its radius. We first establish the required eigenvalue asymptotics for domains with small holes.}

\subsection{Multiple-eigenvalue asymptotics for small holes on manifolds}

\begin{lemma}\label{lem:highdim-manifold-small-holes}
Let $(D,g)$ be a smooth compact $n$-dimensional Riemannian { manifold} with smooth boundary, where $n\ge3$.  Fix distinct interior points
$p_1,\dots,p_N$.  In each tangent space, fix a smooth bounded domain
$\Sigma_\alpha$ containing the origin and having connected exterior.  Set
\[
 { H_{\alpha,\varepsilon}:=\exp_{p_\alpha}(\varepsilon\Sigma_\alpha),\qquad
 D_\varepsilon:=D\setminus\bigcup_\alpha\overline{H_{\alpha,\varepsilon}}.}
\]
Let $\mu$ be a Neumann eigenvalue of $D$ of multiplicity $m$, let
$E(\mu)$ denote its real eigenspace, and let $u_1,\dots,u_m$ be an
$L^2$-orthonormal basis of this eigenspace.  In orthonormal coordinates on the
tangent space, write
$E_\alpha=\mathbb R^n\setminus\overline{\Sigma_\alpha}$ and let $\nu_E$
be the outward unit normal to $E_\alpha$.  For $q\in\mathbb R^n$, let
$W_{\alpha,q}\in\mathcal D^{1,2}(E_\alpha)$ be the unique solution of
    \begin{equation}\label{eq:exterior-problem}
    	-\Delta W_{\alpha,q}=0,\qquad
    	\partial_{\nu_E}W_{\alpha,q}=q\cdot\nu_E
    	\quad\hbox{on }\partial\Sigma_\alpha,
    \end{equation}
where
\[
 \mathcal D^{1,2}(E_\alpha)
 :=\left\{v\in L^{2n/(n-2)}(E_\alpha):
 \nabla v\in L^2(E_\alpha)\right\},\qquad
 \|v\|_{\mathcal D^{1,2}(E_\alpha)}
 :=\|\nabla v\|_{L^2(E_\alpha)},
\]
with the boundary condition understood in the weak sense:
\[
 \int_{E_\alpha}\nabla W_{\alpha,q}\cdot\nabla\varphi\,\dd y
 =\int_{\partial\Sigma_\alpha}
 \varphi(q\cdot\nu_E)\,\dd S,
 \qquad \varphi\in\mathcal D^{1,2}(E_\alpha).
\]
Define
    \begin{equation}\label{eq:T_alpha}
         T_\alpha(q,q')=\int_{E_\alpha}
 \nabla W_{\alpha,q}\cdot\nabla W_{\alpha,q'}\,\dd y.
    \end{equation}
If $\mu=\mu_k(D)=\cdots=\mu_{k+m-1}(D)$, then
\begin{equation}\label{eq:highdim-manifold-hole-branches}
 \mu_{k+i-1}(D_\varepsilon)
 =\mu+\varepsilon^n\eta_i(\mathbf B)+o(\varepsilon^n),\qquad i=1,\ldots,m,
\end{equation}
where { $\eta_i(\mathbf B)$ are eigenvalues of $\mathbf{B}$} listed in nondecreasing order and
\begin{equation}\label{eq:highdim-manifold-hole-matrix}
 \begin{aligned}
 (\mathbf B)_{ij}=\sum_\alpha\Bigl\{&
 |\Sigma_\alpha|\bigl[\mu u_i(p_\alpha)u_j(p_\alpha)
 -\langle\nabla u_i(p_\alpha),\nabla u_j(p_\alpha)\rangle\bigr]\\
 &-T_\alpha(\nabla u_i(p_\alpha),\nabla u_j(p_\alpha))\Bigr\}.
 \end{aligned}
\end{equation}
\end{lemma}

\begin{proof}
We first verify that the { exterior problem \eqref{eq:exterior-problem}} is well posed. 
Let
$B_R$ be a fixed ball containing $\overline{\Sigma_\alpha}$.  For
$v\in\mathcal D^{1,2}(E_\alpha)$, the Sobolev inequality,
H\"older's inequality, and the trace theorem on the fixed domain
$B_R\setminus\overline{\Sigma_\alpha}$ give, successively,
\begin{align*}
 \|v\|_{L^2(\partial\Sigma_\alpha)}
 &\le C\|v\|_{H^1(B_R\setminus\overline{\Sigma_\alpha})}\\
 &\le C\left(
 \|\nabla v\|_{L^2(E_\alpha)}
 +\|v\|_{L^{2n/(n-2)}(E_\alpha)}\right)
 \le C\|\nabla v\|_{L^2(E_\alpha)}.
\end{align*}
Consequently,
\[
 \left|\int_{\partial\Sigma_\alpha}
 v(q\cdot\nu_E)\,\dd S\right|
 \le C|q|\|v\|_{\mathcal D^{1,2}(E_\alpha)},
\]
{which implies 
    \[\int_{\partial\Sigma_\alpha}
    v(q\cdot\nu_E)\,\dd S\]
is a bounded linear functional in $\mathcal D^{1,2}(E_\alpha)$.} The bilinear form $\int_{E_\alpha}\nabla v\cdot\nabla\varphi$ is continuous and coercive,
so the Lax--Milgram theorem yields a unique $W_{\alpha,q}$.

\smallskip
{ \emph{ Step 1: uniform extension and spectral convergence.} We prove the spectral convergence
	\[
	\mu_j(D_\varepsilon)\longrightarrow\mu_j(D)
	\qquad\text{for every fixed }j.
	\]
The spectral convergence argument follows the same min--max and compactness approach as Lemma~\ref{lem:spectral-convergence}, once extension operators with norms independent of $\varepsilon$ are available. The shrinking holes prevent the use of the uniform cone property employed there, so we construct the extensions locally by rescaling fixed reference domains.
	
First, we construct the local extensions. In the tangent space at each $p_\alpha$, choose a fixed Euclidean ball
$B_{R_\alpha}$ such that
$\overline{\Sigma_\alpha}\subset B_{R_\alpha}$.  After decreasing
$\varepsilon_0>0$ if necessary, for every
$0<\varepsilon<\varepsilon_0$ the exponential maps are diffeomorphisms on
$B_{\varepsilon R_\alpha}$ and the corresponding images in the manifold
are pairwise disjoint.  Put
\[
 A_\alpha:=B_{R_\alpha}\setminus\overline{\Sigma_\alpha}.
\]
Since $A_\alpha$ is a fixed smooth domain, there is a bounded linear
extension operator
\[
 P_\alpha:H^1(A_\alpha)\longrightarrow H^1(B_{R_\alpha}),
 \qquad (P_\alpha z)|_{A_\alpha}=z.
\]
For $z\in H^1(A_\alpha)$ with $\int_{A_\alpha}z\,\dd y=0$, the Poincar\'e
inequality on $A_\alpha$ and the boundedness of $P_\alpha$ imply
\begin{equation}\label{eq:reference-extension-gradient-bound}
 \|P_\alpha z\|_{L^2(B_{R_\alpha})}
 +\|\nabla P_\alpha z\|_{L^2(B_{R_\alpha})}
 \le C_\alpha\|\nabla z\|_{L^2(A_\alpha)}.
\end{equation}

Next we define the extension on the whole manifold.  Write
\[
 A_{\alpha,\varepsilon}
 :=\varepsilon A_\alpha
 =B_{\varepsilon R_\alpha}\setminus
   \overline{\varepsilon\Sigma_\alpha}.
\]
 For $v\in H^1(A_{\alpha,\varepsilon})$, let
\[
 m_{\alpha,\varepsilon}(v)
 :=\frac{1}{|A_{\alpha,\varepsilon}|}
   \int_{A_{\alpha,\varepsilon}}v(x)\,\dd x,
 \qquad
 z_{\alpha,\varepsilon}(y)
 :=v(\varepsilon y)-m_{\alpha,\varepsilon}(v).
\]
Then $z_{\alpha,\varepsilon}$ has mean zero on $A_\alpha$.  Define the extension on the rescaled ball by
\[
 (P_{\alpha,\varepsilon}v)(x)
 :=m_{\alpha,\varepsilon}(v)
   +(P_\alpha z_{\alpha,\varepsilon})(x/\varepsilon),
 \qquad x\in B_{\varepsilon R_\alpha}.
\]
It equals $v$ on $A_{\alpha,\varepsilon}$.  Changing variables
$x=\varepsilon y$ in \eqref{eq:reference-extension-gradient-bound} gives
    \begin{align}\label{eq:scaled-extension-gradient-bound}
    	\|\nabla P_{\alpha,\varepsilon}v\|_{L^2(B_{\varepsilon R_\alpha})}^2
    	&=\varepsilon^{n-2}
    	\|\nabla P_\alpha z_{\alpha,\varepsilon}
    	\|_{L^2(B_{R_\alpha})}^2\notag\\
    	&\le C_\alpha\varepsilon^{n-2}
    	\|\nabla z_{\alpha,\varepsilon}\|_{L^2(A_\alpha)}^2\\
    	&=C_\alpha
    	\|\nabla v\|_{L^2(A_{\alpha,\varepsilon})}^2,
    \end{align}
and
    \begin{align}\label{eq:scaled-extension-L2-zero-mean-bound}
    	 \|P_\alpha z_{\alpha,\varepsilon}(\,\cdot\,/\varepsilon)
    	\|_{L^2(B_{\varepsilon R_\alpha})}^2
    	&=\varepsilon^n
    	\|P_\alpha z_{\alpha,\varepsilon}\|_{L^2(B_{R_\alpha})}^2\notag\\
    	&\le C_\alpha\varepsilon^n
    	\|\nabla z_{\alpha,\varepsilon}\|_{L^2(A_\alpha)}^2\\
    	&=C_\alpha\varepsilon^2
    	\|\nabla v\|_{L^2(A_{\alpha,\varepsilon})}^2.
    \end{align}
By the Cauchy--Schwarz inequality, we have
\begin{align}\label{eq:scaled-extension-mean-bound}
 |B_{\varepsilon R_\alpha}|
 |m_{\alpha,\varepsilon}(v)|^2&\leq\frac{\lvert B_{\varepsilon R_{\alpha}}\rvert}{\lvert A_{\alpha,\varepsilon}\rvert}\|v\|_{L^2(A_{\alpha,\varepsilon})}^2
 \le C_\alpha\|v\|_{L^2(A_{\alpha,\varepsilon})}^2.
\end{align}
Consequently,
\begin{equation}\label{eq:scaled-local-extension-L^2-bound}
 \|P_{\alpha,\varepsilon}v\|_{L^2(B_{\varepsilon R_\alpha})}^2
 \le C_\alpha\bigl(
   \|v\|_{L^2(A_{\alpha,\varepsilon})}^2
   +\varepsilon^2
    \|\nabla v\|_{L^2(A_{\alpha,\varepsilon})}^2\bigr),
\end{equation}
and
    \begin{equation}\label{eq:scaled-local-extension-gradient-bound}
    	\|\nabla P_{\alpha,\varepsilon}v\|_{L^2(B_{\varepsilon R_\alpha})}^2
    	\le C_\alpha
    	\|\nabla v\|_{L^2(A_{\alpha,\varepsilon})}^2.
    \end{equation}
Define
\[
 (\mathcal E_\varepsilon v)(q):=
 \begin{cases}
  v(q),&q\in D_\varepsilon,\\
  \bigl(P_{\alpha,\varepsilon}
  (v\circ\exp_{p_\alpha})\bigr)(x),
  &q=\exp_{p_\alpha}(x),\quad x\in B_{\varepsilon R_{\alpha}}.
 \end{cases}
\]
On the overlap domain $A_{\alpha,\varepsilon}=B_{\varepsilon R_\alpha}\setminus
\overline{\varepsilon\Sigma_\alpha}$, we have $P_{\alpha,\varepsilon}(v)=v$, and hence $\mathcal{E}_{\varepsilon}$ is well defined on the whole manifold $D$. Moreover, \eqref{eq:scaled-local-extension-L^2-bound} and \eqref{eq:scaled-local-extension-gradient-bound} and uniform equivalence of the coordinate and Riemannian norms show that it is a bounded linear operator
\[
 \mathcal E_\varepsilon:H^1(D_\varepsilon)\longrightarrow H^1(D),
 \qquad
 \left.\mathcal E_\varepsilon v\right|_{D_\varepsilon}=v,
 \qquad
 \|\mathcal E_\varepsilon v\|_{H^1(D)}
 \le C\|v\|_{H^1(D_\varepsilon)},
\]
where $C$ is independent of $\varepsilon$.

Choose
\[
 0<\gamma<\frac12\operatorname{dist}
 \bigl(\mu,\operatorname{spec}(-\Delta_D^N)\setminus\{\mu\}\bigr).
\]
By spectral convergence, for all sufficiently small $\varepsilon$, the
interval $(\mu-\gamma,\mu+\gamma)$ contains exactly $m$ eigenvalues of
$D_\varepsilon$, and the distance from these eigenvalues to the remaining
eigenvalues is at least $\gamma$.}

\smallskip
{\emph{Step 2: the corrector equation.}
Set
\[
 a_\varepsilon(v,w)=\int_{D_\varepsilon}
 (\langle\nabla v,\nabla w\rangle+vw)\,\dd v_g.
\]
For $u\in E(\mu)$, let $W_\varepsilon^u$ be the unique solution of
\begin{equation}\label{eq:highdim-hole-corrector}
 a_\varepsilon(W_\varepsilon^u,v)
 =\sum_\alpha\int_{\partial H_{\alpha,\varepsilon}}
 v\partial_{\nu_\varepsilon}u\,\dd S_g,
\end{equation}
i.e., the solution of
    \begin{equation}\label{eq:W_epsilon^u}
     \begin{cases}
       -\Delta_g W_{\varepsilon}^u+W_{\varepsilon}^u=0&\text{in }D_{\varepsilon},\\
       \partial_{\nu_{\varepsilon}}W_{\varepsilon}^u=\partial_{\nu_{\varepsilon}}u&\text{on }\bigcup_{\alpha}\partial H_{\alpha,\varepsilon},\\
       \partial_{\nu_g}W_{\varepsilon}^u=0&\text{on }\partial D.
     \end{cases}
    \end{equation}
where $\nu_\varepsilon$ is the outward unit normal to $D_\varepsilon$.
Apply the divergence theorem to $\mathcal E_\varepsilon v$ inside each
hole.  Since the outward normal to a hole is opposite to
$\nu_\varepsilon$, the right-hand side equals
\[
 -\sum_\alpha\int_{H_{\alpha,\varepsilon}}
 \bigl(\langle\nabla u,\nabla\mathcal E_\varepsilon v\rangle
 +(\Delta_gu)\mathcal E_\varepsilon v\bigr)\,\dd v_g.
\]
Set
\[
 H_\varepsilon:=\bigcup_{\alpha=1}^N H_{\alpha,\varepsilon}.
\]
The change of variables $x=\varepsilon y$
therefore gives
\[
 |H_\varepsilon|_g
 =\sum_{\alpha=1}^N|H_{\alpha,\varepsilon}|_g
 \le C\varepsilon^n\sum_{\alpha=1}^N|\Sigma_\alpha|
 \le C\varepsilon^n,
\]
where $|\cdot|_g$ denotes Riemannian volume. The standard elliptic estimates give
\[
 \|\nabla u\|_{L^\infty(D)}
 +\|\Delta_g u\|_{L^\infty(D)}
 \le C\|u\|_{L^2(D)},
 \qquad u\in E(\mu).
\]
Thus, for every $v\in H^1(D_\varepsilon)$,
\begin{align*}
 |a_\varepsilon(W_\varepsilon^u,v)|
 &\le C\|u\|_{L^2(D)}
   \int_{H_\varepsilon}
   \bigl(|\nabla\mathcal E_\varepsilon v|
         +|\mathcal E_\varepsilon v|\bigr)\,\dd v_g\\
 &\le C\|u\|_{L^2(D)}|H_\varepsilon|_g^{1/2}
   \left(\|\nabla\mathcal E_\varepsilon v\|_{L^2(H_\varepsilon)}
         +\|\mathcal E_\varepsilon v\|_{L^2(H_\varepsilon)}\right)\\
 &\le C\varepsilon^{n/2}\|u\|_{L^2(D)}
   \|\mathcal E_\varepsilon v\|_{H^1(D)}\\
 &\le C\varepsilon^{n/2}\|u\|_{L^2(D)}
   \|v\|_{H^1(D_\varepsilon)}.
\end{align*}
Taking $v=W_\varepsilon^u$ in this estimate
and using the definition of $a_\varepsilon$ gives
\[
 \|W_\varepsilon^u\|_{H^1(D_\varepsilon)}^2
 =a_\varepsilon(W_\varepsilon^u,W_\varepsilon^u)
 \le C\varepsilon^{n/2}\|u\|_{L^2(D)}
      \|W_\varepsilon^u\|_{H^1(D_\varepsilon)}.
\]
Hence
\[
 \|W_\varepsilon^u\|_{H^1(D_\varepsilon)}^2
 \le C\varepsilon^n\|u\|_{L^2(D)}^2,
 \qquad u\in E(\mu),
\]
with $C$ independent of $u$ and $\varepsilon$.  In particular,
\begin{equation}\label{eq:highdim-hole-corrector-energy}
 \|W_\varepsilon^u\|_{H^1(D_\varepsilon)}^2=O(\varepsilon^n)
\end{equation}
uniformly
for $u$ on the $L^2(D)$ unit sphere of $E(\mu)$.}

\smallskip
{\emph{Step 3: the $L^2$ mass is of lower order.} We further show that
\begin{equation}\label{eq:highdim-hole-corrector-mass}
 \sup_{\substack{u\in E(\mu)\\ \|u\|_{L^2(D)}=1}}
 \|W_\varepsilon^u\|_{L^2(D_\varepsilon)}^2=o(\varepsilon^n).
\end{equation}
Suppose, to the contrary, that this fails.  Then there exist a constant
$c_*>0$, a sequence $\varepsilon_\ell\downarrow0$, and
$u_\ell\in E(\mu)$ with $\|u_\ell\|_{L^2(D)}=1$ such that
\[
 \varepsilon_\ell^{-n}
 \|W_{\varepsilon_\ell}^{u_\ell}\|_{L^2(D_{\varepsilon_\ell})}^2
 \ge c_*
 \qquad\text{for every }\ell.
\]
Define
\[
 Z_\ell=\varepsilon_\ell^{-n/2}
 \mathcal E_{\varepsilon_\ell}W_{\varepsilon_\ell}^{u_\ell}.
\]
The uniform extension estimate and
\eqref{eq:highdim-hole-corrector-energy} give
\[
 \|Z_\ell\|_{H^1(D)}^2
 \le C\varepsilon_\ell^{-n}
 \|W_{\varepsilon_\ell}^{u_\ell}\|_{H^1(D_{\varepsilon_\ell})}^2
 \le C.
\]
After passing to a subsequence, we assume that
\[
 Z_\ell\rightharpoonup Z\quad\text{in }H^1(D),
 \qquad
 Z_\ell\longrightarrow Z\quad\text{in }L^2(D).
\]

We next identify the limit.  Fix $\varphi\in H^1(D)$ that vanishes in
a neighborhood of each $p_\alpha$.  All the holes lie in these
neighborhoods for sufficiently large $\ell$, so the trace of $\varphi$
on every hole boundary is zero.  Testing
\eqref{eq:highdim-hole-corrector} with
$v=\varphi|_{D_{\varepsilon_\ell}}$ and multiplying by
$\varepsilon_\ell^{-n/2}$ gives
\[
 \int_{D_{\varepsilon_\ell}}
 \bigl(\langle\nabla Z_\ell,\nabla\varphi\rangle
       +Z_\ell\varphi\bigr)\,\dd v_g=0.
\]
Both $\varphi$ and its weak gradient vanish
inside the holes, so the integral is unchanged if
$D_{\varepsilon_\ell}$ is replaced by $D$.  Passing to the weak
limit thus gives
\[
 \int_D\bigl(\langle\nabla Z,\nabla\varphi\rangle+Z\varphi\bigr)
 \,\dd v_g=0.
\]
Since $n\geq3$, a finite set of points has zero $H^1$ capacity. Thus functions vanishing near all $p_\alpha$ are dense in $H^1(D)$, and the identity extends to every $\varphi\in H^1(D)$. Taking $\varphi=Z$ yields
\[
 \int_D\bigl(|\nabla Z|^2+|Z|^2\bigr)\,\dd v_g=0,
\]
and hence $Z=0$.

The strong $L^2(D)$ convergence now implies
\[
 c_*
 \le\varepsilon_\ell^{-n}
     \|W_{\varepsilon_\ell}^{u_\ell}\|_{L^2(D_{\varepsilon_\ell})}^2
 =\|Z_\ell\|_{L^2(D_{\varepsilon_\ell})}^2
 \le\|Z_\ell\|_{L^2(D)}^2
 \longrightarrow0,
\]
a contradiction.}

{
 \smallskip
\emph{Step 4: leading-order corrector energy.} By Step 2, we know
    \[a_{\varepsilon}(W_{\varepsilon}^u,W_{\varepsilon}^v)=O(\varepsilon^n)\]
for fixed $u,v\in E(\mu)$. In this step we identify the coefficient of the $\varepsilon^n$ term:
    \begin{equation}\label{eq:highdim-hole-corrector-limit}
    	a_{\varepsilon}(W_{\varepsilon}^u,W_{\varepsilon}^v)=\varepsilon^n\sum_{\alpha}T_{\alpha}(\nabla u(p_{\alpha}),\nabla v(p_{\alpha}))+o(\varepsilon^n),
    \end{equation}
where $T_{\alpha}$ is defined by \eqref{eq:T_alpha}. The remainder is uniform
for $u,v$ on the unit sphere.

Fix a sufficiently small radius $\rho>0$.  Each map $\exp_{p_\alpha}$
is then a diffeomorphism on a neighborhood of $\overline{B_\rho(0)}$,
and the images of these closed balls are pairwise disjoint and contained
in the interior of $D$. Fix an index $\alpha$, $u\in E(\mu)$, and write
\[
 F_\varepsilon(y):=\exp_{p_\alpha}(\varepsilon y),
 \qquad
 U_{\alpha,\varepsilon}
 :=B_{\rho/\varepsilon}(0)\setminus\overline{\Sigma_\alpha}.
\]
For all sufficiently small $\varepsilon$,
$\varepsilon\overline{\Sigma_\alpha}\subset B_\rho(0)$ and
\[
 F_\varepsilon(U_{\alpha,\varepsilon})
 =D_\varepsilon\cap\exp_{p_\alpha}(B_\rho(0)).
\]
Define
\[
 \widehat W_{\alpha,\varepsilon}^u(y)
 :=\varepsilon^{-1}W_\varepsilon^u(F_\varepsilon(y)),
 \qquad y\in U_{\alpha,\varepsilon}.
\]

The energy estimate
\eqref{eq:highdim-hole-corrector-energy}, the uniform extension estimate,
and the Sobolev inequality on $D$, with $2^*=2n/(n-2)$, give
    \[\int_{U_{\alpha,\varepsilon}}
    |\nabla_y\widehat W_{\alpha,\varepsilon}^u|^2\,\dd y\le C\varepsilon^{-n}
    \|W_\varepsilon^u\|_{H^1(D_\varepsilon)}^2
    \le C\|u\|_{L^2(D)}^2,\]
and
    \[\|\widehat W_{\alpha,\varepsilon}^u\|_{L^{2^*}(U_{\alpha,\varepsilon})}\le C\varepsilon^{-n/2}\|\mathcal E_\varepsilon W_\varepsilon^u\|_{L^{2^*}(D)}
    \le C\|u\|_{L^2(D)}.\]
For each fixed $R$ with $\overline{\Sigma_\alpha}\subset B_R(0)$, put
\[
 E_{\alpha,R}:=B_R(0)\setminus\overline{\Sigma_\alpha}.
\]
Then $E_{\alpha,R}\subset U_{\alpha,\varepsilon}$ whenever
$\varepsilon<\rho/R$, and H\"older's inequality gives
\[
 \|\widehat W_{\alpha,\varepsilon}^u\|_{L^2(E_{\alpha,R})}
 \le |E_{\alpha,R}|^{1/n}
       \|\widehat W_{\alpha,\varepsilon}^u\|_{L^{2^*}(U_{\alpha,\varepsilon})}.
\]
Thus $\{\widehat{W}_{\alpha,\varepsilon}^u\}$ is bounded in $H^1(E_{\alpha,R})$.  From any sequence
$\varepsilon\downarrow0$ we may extract a diagonal subsequence such that
\[
 \widehat W_{\alpha,\varepsilon}^u\rightharpoonup W_*
 \quad\text{in }H^1(E_{\alpha,R})
 \quad\text{for every fixed}\ R\ \text{with}\ \overline{\Sigma_\alpha}\subset B_{R}.
\]
Letting $R\to\infty$ and applying weak lower semicontinuity to the gradient and $L^{2^*}$ norms, we obtain $W_*\in\mathcal D^{1,2}(E_\alpha)$.

Let $\varphi\in C_c^\infty(\overline{E_\alpha})$, and choose $R$ such
that $\overline{\Sigma_\alpha}$ and $\operatorname{supp}\varphi$
are contained in $B_R(0)$.  For $\varepsilon<\rho/R$, define
\[
 v_\varepsilon(F_\varepsilon(y))=\varepsilon\varphi(y),
 \qquad y\in U_{\alpha,\varepsilon},
\]
and set $v_\varepsilon=0$ on the remainder of $D_\varepsilon$.
Then $v_{\varepsilon}\in H^1(D_\varepsilon)$. Using this test function in~\eqref{eq:highdim-hole-corrector} and dividing
by $\varepsilon^n$, we obtain
    \begin{equation}\label{eq:test-with-v_epsilon}
    	\int_{E_{\alpha,R}}\sum_{i,j=1}^n
    	g^{ij}(\varepsilon y)
    	\partial_i\widehat W_{\alpha,\varepsilon}^u
    	\partial_j\varphi\sqrt{\det G_\varepsilon}\,\dd y+\varepsilon^2\int_{E_{\alpha,R}}
    	\widehat W_{\alpha,\varepsilon}^u\varphi
    	\sqrt{\det G_\varepsilon}\,\dd y=\int_{\partial\Sigma_\alpha}
    	\varphi b_\varepsilon^u J_\varepsilon\,\dd S,
    \end{equation}
where $G_\varepsilon(y)=(g_{ij}(\varepsilon y))$ is the metric matrix in normal coordinates, $(g^{ij}(\varepsilon y))=G_\varepsilon(y)^{-1}$, $J_{\varepsilon}=1+O(\varepsilon^2)$ is the surface Jacobian after extracting the factor $\varepsilon^{n-1}$, and
    \[b_\varepsilon^u(y)=\partial_{\nu_\varepsilon}u(F_\varepsilon(y))=\nabla u(p_\alpha)\cdot\nu_E(y)+O(\varepsilon).\]
Consequently,
\[
 \int_{E_\alpha}\nabla W_*\cdot\nabla\varphi\,\dd y
 =\int_{\partial\Sigma_\alpha}
   \varphi\bigl(\nabla u(p_\alpha)\cdot\nu_E\bigr)\,\dd S.
\]
This implies that $W_*$ solves the
exterior problem~\eqref{eq:exterior-problem} with
$q=\nabla u(p_\alpha)$.  Uniqueness identifies
$W_*=W_{\alpha,\nabla u(p_\alpha)}$ and hence
    \begin{equation}\label{eq:weak-limit-W_*}
    	\widehat W_{\alpha,\varepsilon}^u
    	\rightharpoonup W_{\alpha,\nabla u(p_\alpha)}
    	\quad\text{in }H^1(E_{\alpha,R})
    	\quad\text{for every fixed }R
    	\text{ with }\overline{\Sigma_\alpha}\subset B_R.
    \end{equation}

Return to all the holes, and restore the index $\alpha$ in the notation:
write $F_{\alpha,\varepsilon}(y)=\exp_{p_\alpha}(\varepsilon y)$,
and let $J_{\alpha,\varepsilon}$ and $b_{\alpha,\varepsilon}^u$
denote the quantities in \eqref{eq:test-with-v_epsilon}.
For another $v\in E(\mu)$, take
$W_\varepsilon^v$ as the test function
in~\eqref{eq:highdim-hole-corrector}.  This gives
\[
 a_\varepsilon(W_\varepsilon^u,W_\varepsilon^v)
 =\sum_\alpha\int_{\partial H_{\alpha,\varepsilon}}
 W_\varepsilon^v\partial_{\nu_\varepsilon}u\,\dd S_g.
\]
On each boundary,
$W_\varepsilon^v(F_{\alpha,\varepsilon}(y))
=\varepsilon\widehat W_{\alpha,\varepsilon}^v(y)$ in the trace sense.
The change of variables gives
\[
 \varepsilon^{-n}a_\varepsilon(W_\varepsilon^u,W_\varepsilon^v)
 =\sum_\alpha\int_{\partial\Sigma_\alpha}
 \widehat W_{\alpha,\varepsilon}^v
 b_{\alpha,\varepsilon}^u J_{\alpha,\varepsilon}\,\dd S.
\]
By \eqref{eq:weak-limit-W_*} and continuity of the trace operator,
\[
 \widehat W_{\alpha,\varepsilon}^v\big|_{\partial\Sigma_\alpha}
 \rightharpoonup
 W_{\alpha,\nabla v(p_\alpha)}\big|_{\partial\Sigma_\alpha}
 \quad\text{in }L^2(\partial\Sigma_\alpha).
\]
Thus,
\begin{equation}
 \varepsilon^{-n}a_\varepsilon(W_\varepsilon^u,W_\varepsilon^v)
 \longrightarrow
 \sum_\alpha\int_{\partial\Sigma_\alpha}
 W_{\alpha,\nabla v(p_\alpha)}
 (\nabla u(p_\alpha)\cdot\nu_E)\,\dd S
 =\sum_\alpha
 T_\alpha(\nabla u(p_\alpha),\nabla v(p_\alpha)).
\end{equation}
The last equality follows by testing the exterior equation for
$W_{\alpha,\nabla u(p_\alpha)}$ with
$W_{\alpha,\nabla v(p_\alpha)}$. Since $u\mapsto W_\varepsilon^u$ and $q\mapsto W_{\alpha,q}$ are linear, convergence for the finitely many pairs of basis elements of $E(\mu)$ implies convergence in the norm of bilinear forms. This proves the asserted uniformity for $u,v$ on the unit sphere.}

\smallskip
{\emph{Step 5: complete the proof.}
Let $P_\varepsilon u=u-W_\varepsilon^u$ and define
\[
 q_\varepsilon(v,w):=a_\varepsilon(v,w)
 -(\mu+1)(v,w)_{L^2(D_\varepsilon)}.
\]
Since $-\Delta_g u=\mu u$ in $D$ and $\partial_\nu u=0$ on
$\partial D$, Green's formula on $D_\varepsilon$ gives, for
$u\in E(\mu)$ and $v\in H^1(D_\varepsilon)$,
    \begin{align}\label{eq:q_epsilon-u-v}
    	q_\varepsilon(u,v)
    	&=\int_{D_\varepsilon}
    	(\langle\nabla u,\nabla v\rangle-\mu uv)\,\dd v_g\\
    	&=\sum_\alpha\int_{\partial H_{\alpha,\varepsilon}}
    	v\partial_{\nu_\varepsilon}u\,\dd S_g
    	=a_\varepsilon(W_\varepsilon^u,v).
    \end{align}
The last equality is the defining equation~\eqref{eq:highdim-hole-corrector}.
As
    \[q_\varepsilon(W_\varepsilon^u,v)
    =a_\varepsilon(W_\varepsilon^u,v)
    -(\mu+1)(W_\varepsilon^u,v)_{L^2(D_\varepsilon)},\]
we obtain
\begin{equation}\label{eq:highdim-hole-residual}
 q_\varepsilon(P_\varepsilon u,v)
 =(\mu+1)(W_\varepsilon^u,v)_{L^2(D_\varepsilon)},
\end{equation}
for every $v\in H^1(D_\varepsilon)$.

For $u,v\in E(\mu)$, by \eqref{eq:q_epsilon-u-v}, we have
\[
 q_\varepsilon(u,W_\varepsilon^v)
 =a_\varepsilon(W_\varepsilon^u,W_\varepsilon^v)
 =q_\varepsilon(W_\varepsilon^u,v).
\]
Here the second equality uses the symmetry of $q_\varepsilon$ and
$a_\varepsilon$.  Expanding $q_\varepsilon(P_\varepsilon u,P_\varepsilon v)$ therefore yields
\[
 \begin{aligned}
 q_\varepsilon(P_\varepsilon u,P_\varepsilon v)
 &=q_\varepsilon(u,v)-q_\varepsilon(u,W_\varepsilon^v)
   -q_\varepsilon(W_\varepsilon^u,v)
   +q_\varepsilon(W_\varepsilon^u,W_\varepsilon^v)\\
 &=q_\varepsilon(u,v)
   -a_\varepsilon(W_\varepsilon^u,W_\varepsilon^v)
   -(\mu+1)(W_\varepsilon^u,W_\varepsilon^v)_{L^2(D_\varepsilon)}.
 \end{aligned}
\]
Since $u,v\in E(\mu)$, we have
    \[\int_D(\langle\nabla u,\nabla v\rangle-\mu uv)\,\dd v_g=0.\]
Therefore
\begin{align}\label{eq:highdim-hole-restricted-form}
 q_\varepsilon(P_\varepsilon u,P_\varepsilon v)
 ={}&-\int_{\bigcup_\alpha H_{\alpha,\varepsilon}}
  (\langle\nabla u,\nabla v\rangle-\mu uv)\,\dd v_g
 -a_\varepsilon(W_\varepsilon^u,W_\varepsilon^v)\notag\\
 &-(\mu+1)(W_\varepsilon^u,W_\varepsilon^v)_{L^2(D_\varepsilon)}.
\end{align}
A Taylor expansion in normal coordinates at each $p_\alpha$ gives
\begin{align}\label{eq:int-on-H_alpha_epsilon}
 &-\int_{\bigcup_\alpha H_{\alpha,\varepsilon}}
 (\langle\nabla u,\nabla v\rangle-\mu uv)\,\dd v_g\\
 &\qquad=\varepsilon^n\sum_\alpha|\Sigma_\alpha|
 \bigl[\mu u(p_\alpha)v(p_\alpha)
 -\langle\nabla u(p_\alpha),\nabla v(p_\alpha)\rangle\bigr]
 +o(\varepsilon^n),
\end{align}
uniformly for $u,v$ on the $L^2(D)$ unit sphere of $E(\mu)$. Now, by \eqref{eq:highdim-hole-corrector-mass}, \eqref{eq:highdim-hole-corrector-limit}, \eqref{eq:highdim-hole-restricted-form} and \eqref{eq:int-on-H_alpha_epsilon}, we obtain
\begin{equation}\label{eq:highdim-hole-effective-form}
 q_\varepsilon(P_\varepsilon u,P_\varepsilon v)
 =\varepsilon^n\mathbf B(u,v)+o(\varepsilon^n),
\end{equation}
where $\mathbf B$ is the bilinear form with matrix
\eqref{eq:highdim-manifold-hole-matrix} in the basis
$u_1,\ldots,u_m$, and the remainder is uniform on the unit sphere.
Furthermore, $\int_{D_\varepsilon}u_i u_j\,\dd v_g
=\delta_{ij}+O(\varepsilon^n)$, by
\eqref{eq:highdim-hole-corrector-mass},
    \begin{equation}\label{eq:P_epsilonu_i-P_epsilonu_j}
    	 \bigl((P_\varepsilon u_i,P_\varepsilon u_j)_{L^2(D_\varepsilon)}\bigr)_{i,j=1}^m
    	=I_m+o(1).
    \end{equation}

Write $A_\varepsilon=-\Delta_{D_\varepsilon}^N$ for the self-adjoint Neumann Laplacian associated with the Dirichlet form on $H^1(D_\varepsilon)$. The weak identity \eqref{eq:highdim-hole-residual} shows that $P_\varepsilon u\in\mathcal D(A_\varepsilon)$ and
\[
 (A_\varepsilon-\mu)P_\varepsilon u
 =r_\varepsilon^u:=(\mu+1)W_\varepsilon^u.
\]
\eqref{eq:highdim-hole-corrector-mass} now gives
\[
 \sup_{\substack{u\in E(\mu)\\ \|u\|_{L^2(D)}=1}}
 \|r_\varepsilon^u\|_{L^2(D_\varepsilon)}
 =o(\varepsilon^{n/2}).
\]
Let
$\mathcal C_\varepsilon$ be the spectral subspace spanned by eigenfunctions whose
 eigenvalues lie in $(\mu-\gamma,\mu+\gamma)$. By the last paragraph of Step 1, $\dim\mathcal C_\varepsilon=m$ for small $\varepsilon$. Let $\Pi_\varepsilon$ be the $L^2(D_\varepsilon)$-orthogonal projection
onto $\mathcal C_\varepsilon$ and let $Q_\varepsilon=I-\Pi_\varepsilon$.  Every eigenvalue of
$A_\varepsilon$ on $\operatorname{ran}Q_\varepsilon$ satisfies
$|\lambda-\mu|\ge\gamma$.  Since $Q_\varepsilon$ commutes with
$A_\varepsilon$, the spectral theorem gives
\[
 \|Q_\varepsilon f\|_{L^2(D_\varepsilon)}
 \le\gamma^{-1}
 \|(A_\varepsilon-\mu)Q_\varepsilon f\|_{L^2(D_\varepsilon)},
 \qquad f\in\mathcal D(A_\varepsilon).
\]
Applying this to $P_\varepsilon u$ and using
$(A_\varepsilon-\mu)Q_\varepsilon P_\varepsilon u
=Q_\varepsilon r_\varepsilon^u$ gives
\begin{align*}
 \|Q_\varepsilon P_\varepsilon u\|_{L^2(D_\varepsilon)}
 &\le\gamma^{-1}\|r_\varepsilon^u\|_{L^2(D_\varepsilon)}
 =o(\varepsilon^{n/2}).
\end{align*}
Since
$q_\varepsilon(f,g)=((A_\varepsilon-\mu)f,g)_{L^2(D_\varepsilon)}$
for $f\in\mathcal D(A_\varepsilon)$ and $g\in H^1(D_\varepsilon)$, we have
\begin{align}\label{eq:q_epsilon-Q_epsilonP_epsilonu}
 \left|q_\varepsilon(Q_\varepsilon P_\varepsilon u,
 Q_\varepsilon P_\varepsilon v)\right|
 &=\left|(Q_\varepsilon r_\varepsilon^u,
   Q_\varepsilon P_\varepsilon v)_{L^2(D_\varepsilon)}\right|\\
 &\le\gamma^{-1}
 \|r_\varepsilon^u\|_{L^2(D_\varepsilon)}
 \|r_\varepsilon^v\|_{L^2(D_\varepsilon)}
 =o(\varepsilon^n),
\end{align}
uniformly for $u,v$ on the $L^2(D)$ unit sphere of $E(\mu)$.

Let $J_\varepsilon=\Pi_\varepsilon P_\varepsilon:E(\mu)\to
\mathcal C_\varepsilon$.  The $L^2$-orthogonal decomposition
$P_\varepsilon u=J_\varepsilon u+Q_\varepsilon P_\varepsilon u$ shows that
\[
 \begin{aligned}
 (J_\varepsilon u_i,J_\varepsilon u_j)_{L^2(D_\varepsilon)}=(P_\varepsilon u_i,P_\varepsilon u_j)_{L^2(D_\varepsilon)}
 -(Q_\varepsilon P_\varepsilon u_i,
   Q_\varepsilon P_\varepsilon u_j)_{L^2(D_\varepsilon)}.
 \end{aligned}
\]
Consequently, by \eqref{eq:P_epsilonu_i-P_epsilonu_j} and \eqref{eq:q_epsilon-Q_epsilonP_epsilonu},
\[
 \bigl((J_\varepsilon u_i,J_\varepsilon u_j)_{L^2(D_\varepsilon)}\bigr)_{i,j=1}^m
 =I_m+o(1).
\]
So $J_\varepsilon$ is injective for small $\varepsilon$.  Both $E(\mu)$ and
$\mathcal C_\varepsilon$ have dimension $m$; hence $J_\varepsilon$
is an isomorphism onto $\mathcal C_\varepsilon$, and $J_{\varepsilon}u_{1},\cdots,J_{\varepsilon}u_{m}$ is a basis for $\mathcal{C}_{\varepsilon}$.

The spectral decomposition is also orthogonal for $q_\varepsilon$.
Indeed, if $x\in\mathcal C_\varepsilon$ and
$y\in\mathcal D(A_\varepsilon)\cap\mathcal C_\varepsilon^\perp$,
then $(A_\varepsilon-\mu)x\in\mathcal C_\varepsilon$ and thus
$q_\varepsilon(x,y)=((A_\varepsilon-\mu)x,y)_{L^2(D_\varepsilon)}=0$.
Therefore
    \[q_\varepsilon(P_\varepsilon u,P_\varepsilon v)=q_\varepsilon(J_\varepsilon u,J_\varepsilon v)+q_\varepsilon(Q_\varepsilon P_\varepsilon u,
    Q_\varepsilon P_\varepsilon v).\]
Combining this identity with~\eqref{eq:highdim-hole-effective-form}
and \eqref{eq:q_epsilon-Q_epsilonP_epsilonu} gives
\[
 q_\varepsilon(J_\varepsilon u,J_\varepsilon v)
 =\varepsilon^n\mathbf B(u,v)+o(\varepsilon^n),
\]
uniformly on the unit sphere of $E(\mu)$.

To obtain \eqref{eq:highdim-manifold-hole-branches}, put $e_i=J_\varepsilon u_i$ for $i=1,\ldots,m$. They form a basis of
$\mathcal C_\varepsilon$.  Define
\[
 F_\varepsilon:=
 \bigl(q_\varepsilon(e_i,e_j)\bigr)_{i,j=1}^m,
 \qquad
 G_\varepsilon:=
 \bigl((e_i,e_j)_{L^2(D_\varepsilon)}\bigr)_{i,j=1}^m.
\]
Suppose $\phi=\sum_{j=1}^m c_j e_j\ne0$ is an eigenfunction of
$A_\varepsilon$ in $\mathcal C_\varepsilon$, with eigenvalue
$\lambda$.  For every $i$, we have
\[
 \begin{aligned}
 (F_\varepsilon c)_i
 &=q_\varepsilon(e_i,\phi)
  =((A_\varepsilon-\mu)\phi,e_i)_{L^2(D_\varepsilon)}\\
 &=(\lambda-\mu)(\phi,e_i)_{L^2(D_\varepsilon)}
  =(\lambda-\mu)(G_\varepsilon c)_i.
 \end{aligned}
\]
Conversely, if $F_\varepsilon c=\theta G_\varepsilon c$ with
$c\ne0$ and $\phi=\sum_jc_j e_j$, then
$((A_\varepsilon-\mu)\phi-\theta\phi,e_i)_{L^2(D_\varepsilon)}=0$
for every $i$.  The vector
$(A_\varepsilon-\mu)\phi-\theta\phi$ belongs to
$\mathcal C_\varepsilon$, because this is a spectral subspace of
$A_\varepsilon$.  As the $e_i$ form a basis of that subspace, the
vector must vanish.  Thus the generalized eigenvalues $\theta$ are
exactly the eigenvalues of $A_\varepsilon-\mu$ on
$\mathcal C_\varepsilon$, counted with multiplicity:
    \begin{equation}\label{eq:eigenvalue-theta}
    	\theta=\mu_{k+i-1}(D_\varepsilon)-\mu,
    	\qquad i=1,\ldots,m,
    \end{equation}
in nondecreasing order.

It remains to estimate these generalized eigenvalues. By \eqref{eq:highdim-hole-effective-form}, we have
\[
 F_\varepsilon=\varepsilon^n\mathbf B+o(\varepsilon^n),
 \qquad G_\varepsilon=I_m+o(1).
\]
In particular, $G_\varepsilon$ is positive definite and
$G_\varepsilon^{-1/2}=I_m+o(1)$.  The generalized eigenvalues of
$F_\varepsilon c=\theta G_\varepsilon c$ are the ordinary
eigenvalues of the symmetric matrix
\[
 G_\varepsilon^{-1/2}F_\varepsilon G_\varepsilon^{-1/2}
 =\varepsilon^n\mathbf B+o(\varepsilon^n).
\]
The min--max principle for symmetric matrices shows that each ordered
eigenvalue changes by at most the operator norm of the error matrix.
Hence the $i$th generalized eigenvalue equals
$\varepsilon^n\eta_i(\mathbf B)+o(\varepsilon^n)$, including when
$\mathbf B$ has repeated eigenvalues.  Combining this estimate with \eqref{eq:eigenvalue-theta} yields, for every $i=1,\ldots,m$,
\[
 \mu_{k+i-1}(D_\varepsilon)
 =\mu+\varepsilon^n\eta_i(\mathbf B)+o(\varepsilon^n),
\]
which is~\eqref{eq:highdim-manifold-hole-branches}.}
\end{proof}

\begin{remark}
Felli--Liverani--Ognibene
\cite[Theorem~1.3 and equation~(1.8)]{FelliLiveraniOgnibene2026}
{established} a multiple-eigenvalue expansion for a single small hole in
Euclidean space. {Since their operator} is $-\Delta+1$, {the quantity} 
$\lambda-1$ {in their work} corresponds to our $\mu$.  {Let $\mathcal L$ denote the
bilinear form in their equation~(1.8). In the Euclidean single-hole case, with the identification
$\lambda=\mu+1$, we have $\mathbf B=-\mathcal L$. Lemma~\ref{lem:highdim-manifold-small-holes} adapts this framework to finitely many holes on a Riemannian manifold; that extension is established by the proof above.}  Consequently, if
$\ell_1\ge\cdots\ge\ell_m$ are the eigenvalues of $\mathcal L$ in
nonincreasing order, then the eigenvalues of $\mathbf B$ here, in
nondecreasing order, satisfy
\[
 \eta_i(\mathbf B)=-\ell_i,
 \qquad i=1,\ldots,m.
\]
{The eigenvalues in \cite{FelliLiveraniOgnibene2026} are indexed starting from $0$, whereas ours are indexed starting from $1$. Thus, on the same Euclidean domain, their $\lambda_j$ corresponds to our $\mu_{j+1}+1$.}
\end{remark}

\subsection{Polarization coefficients of oblate ellipsoids}

For an integral representation of the Newtonian potential of an
ellipsoid, see \cite[Theorem~2.1]{DiFratta2016}. {For every \(n\geq3\), we construct an explicit solution of the exterior Neumann problem \eqref{eq:exterior-problem} for each principal-axis direction, verify that it belongs to \(\mathcal{D}^{1,2}(E)\), and compute the associated bilinear energies \(T\) defined in \eqref{eq:T_alpha}.}

\begin{lemma}\label{lem:highdim-oblate-polarization}
{For $n\geq3$ and $a_1,\,a_2,\,\cdots,\,a_n>0$, let
\begin{equation}
\Sigma=\left\{x\in\mathbb{R}^n:\sum_{i=1}^n \frac{x_i^2}{a^2_i}<1\right\}.
\end{equation}
 For $1\leq i\leq n$, define
\[
 L_i=\frac{\prod^n_{j=1}  a_j}{2}\int_0^\infty
 \frac{\dd s}{(a_i^2+s)\sqrt{\prod^n_{j=1} (a_j^2+s)}}.
\]
Then we have $0<L_i<1$, $\sum^n_{i=1}  L_i=1$. Moreover, let \(e_1,\,\cdots,\,e_n\) be the standard orthonormal basis. The exterior energy $T$ defined by \eqref{eq:T_alpha}
satisfies
\[
 T(e_i,e_j)=\delta_{ij}|\Sigma|\frac{L_i}{1-L_i},\qquad 1\leq i,\,j\leq n.
\]
In particular, suppose that $a_n=\tau\in(0,\,1)$ and $a_1=a_2=\cdots=a_{n-1}=1$, then we have
\[
 L(\tau)\coloneqq L_1=L_2=\cdots=L_{n-1}=\frac\tau2\int_0^\infty
 \frac{\dd s}{(1+s)^{(n+1)/2}(\tau^2+s)^{1/2}}
 =O(\tau)\longrightarrow0\quad\text{as }\tau\to0^+.
\]}
\end{lemma}

\begin{proof}
Write $A=\prod_{j=1}^na_j$. {For each
$x\in E:=\mathbb R^n\setminus\overline\Sigma$, the function
\[
 r\longmapsto\sum_{j=1}^n\frac{x_j^2}{a_j^2+r}
\]
is strictly decreasing on $[0,\infty)$, starting from a value greater than $1$ at $r=0$ and tending to 0 as $r\to+\infty$.}  Hence there is a unique $\rho(x)>0$ such that
\begin{equation}
\label{confocal equation}
 \sum_{j=1}^n\frac{x_j^2}{a_j^2+\rho(x)}=1.
\end{equation}
{Set $\rho=0$ continuously on $\partial\Sigma$.  Since the derivative of
the left-hand side of \eqref{confocal equation}  with respect to $\rho$ is strictly negative, the
implicit function theorem shows that $\rho$ is smooth in $E$ and has a smooth extension to the boundary $\partial E$}.

Define
\[
 A_j=a_j^2+\rho,\qquad
 S=\sum_{j=1}^n\frac{x_j^2}{A_j^2},\qquad
 H=\sum_{j=1}^n\frac1{A_j},\qquad
 K=\sum_{j=1}^n\frac{x_j^2}{A_j^3}.
\]
Differentiating \eqref{confocal equation} with respect to $x_i$ gives
\begin{equation}
\label{nabla of rho}
 \partial_i\rho=\frac{2x_i}{A_iS},
 \qquad
 |\nabla\rho|^2=\frac4S.
\end{equation}
We also have
\[
 \partial_iS=\frac{2x_i}{A_i^2}-2K\partial_i\rho,
 \qquad
 \sum_{i=1}^n\frac{x_i\partial_i\rho}{A_i}=2,
 \qquad
 \sum_{i=1}^n\frac{x_i\partial_iS}{A_i}=-2K.
\]
{Thus, we obtain
\begin{align}
 \Delta\rho
 &=\frac{2H}{S}
 -\frac2S\sum_{i=1}^n\frac{x_i\partial_i\rho}{A_i^2}
 -\frac2{S^2}\sum_{i=1}^n\frac{x_i\partial_iS}{A_i}\nonumber\\
 &=\frac{2H}{S}-\frac{4K}{S^2}+\frac{4K}{S^2}\nonumber\\
 &=\frac{2H}{S}.\label{laplacian of rho}
\end{align}
For $1\leq i\leq n$, define
\begin{equation}
 I_i(r)\coloneqq\int_r^\infty
 \frac{\dd s}{(a_i^2+s)\sqrt{\prod_{j=1}^n(a_j^2+s)}},
 \qquad
 V_i(x)\coloneqq x_iI_i(\rho(x)).
 \label{definition of Ii}
\end{equation}
Direct computation shows that
\[
 I_i''(\rho)
 =-\left(\frac1{A_i}+\frac H2\right)I_i'(\rho).
\]
Combining \eqref{nabla of rho} and \eqref{laplacian of rho}, we obtain
\begin{align*}
 \Delta V_i
 &=x_i\left(I_i''|\nabla\rho|^2+I_i'\Delta\rho\right)
 +2I_i'\partial_i\rho\\
 &=\frac{x_i}S
 \left(4I_i''+2HI_i'+\frac4{A_i}I_i'\right)\\
 &=0.
\end{align*}
Thus $V_i$ is harmonic in $E$.}

{We next verify that $V_i\in\mathcal D^{1,2}(E)$.  Let 
\[a_{\min}=\min_{1\leq j\leq n}a_j , \qquad a_{\max}=\max_{1\leq j\leq n}a_j.
\]
Replacing every denominator $a_j^2+\rho$ in \eqref{confocal equation} by, respectively, $a_{\max}^2+\rho$ and $a_{\min}^2+\rho$, gives
\[
 |x|^2-a_{\max}^2\le\rho(x)\le |x|^2-a_{\min}^2.
\]
We can choose $R>0$ sufficiently large such that
\begin{equation}
   1< c|x|^2\leq \rho(x)\leq C|x|^2\qquad\text{for }|x|>R,
   \label{estimate of rho} 
\end{equation}
where $c$ and $C$ are positive constants independent of $x$. Since $|\nabla\rho|^2=4/S$, it follows that
\begin{equation}
S(x)=\sum_{i=1}^n\frac{x_i^2}{A_i^2}\geq c_1|x|^{-2}\qquad\text{and} \qquad |\nabla \rho(x) |\leq C_1|x|\qquad \text{for}\quad |x|\geq R.
\label{estimate of nabla rho}
\end{equation}
On the other hand, by the definition of $I_i$ in \eqref{definition of Ii}, for $r>1$, we have 
\begin{equation}
\label{estimate of Ii} 
     |I_i(r)|\leq C_2 r^{-\frac{n}{2}},\qquad |I'_i(r)|\leq C_2 r^{-\frac{n}{2}-1}
\end{equation}
Since $V_i=x_iI_i(\rho)$, combining \eqref{estimate of rho}, \eqref{estimate of nabla rho} and \eqref{estimate of Ii}, we have
\begin{equation}
\label{estimate of Vi} 
 |V_i(x)|\le C_3|x|^{1-n},
 \qquad
 |\nabla V_i(x)|\le C_3|x|^{-n}\qquad\text{for}\quad |x|\geq R.
\end{equation}
Therefore, for sufficiently large $R$,
\begin{align*}
 \int_{E\cap\{|x|>R\}}|\nabla V_i|^2\,\dd x
 &\le C\int_R^\infty r^{-n-1}\,\dd r<\infty,\\
 \int_{E\cap\{|x|>R\}}|V_i|^{2n/(n-2)}\,\dd x
 &\le C\int_R^\infty
 r^{n-1-\frac{2n(n-1)}{(n-2)}}\,\dd r<\infty.
\end{align*}
The function $V_i$ is smooth in a fixed neighborhood of the boundary;
hence $V_i\in\mathcal D^{1,2}(E)$.}

We now prove that $0<L_i<1$.  Let
\[
 G(s)=\prod_{j=1}^n(a_j^2+s)^{-1/2}.
\]
Then
\[
 G'(s)=-\frac12G(s)\sum_{i=1}^n\frac1{a_i^2+s}.
\]
Therefore,
\begin{align*}
 \sum_{i=1}^nL_i
 &=-A\int_0^\infty G'(s)\,\dd s
 =A\bigl(G(0)-G(\infty)\bigr)=1.
\end{align*}
{  Since each $L_i>0$, we obtain that $0<L_i<1$.}

{ Let $\nu_\Sigma$ denote the outward unit normal vector on $\partial \Sigma$. Then we have
\begin{equation*} 
 \nu_\Sigma=\frac{\left(x_1/a_1^2,\, x_2/a_2^2,\,\cdots,\, x_n/a_n^2\right)}{\sqrt{S_0}}\qquad\text{where}\quad S_0=\sum_{j=1}^n\frac{x_j^2}{a_j^4}.
\end{equation*}
It follows that
\[
 I_i(0)=\frac{2L_i}{A},\qquad
 I_i'(0)=-\frac1{a_i^2A},\qquad
 \partial_{\nu_\Sigma}\rho=\frac2{\sqrt{S_0}}.
\]
Consequently,
\begin{align*}
 \partial_{\nu_\Sigma}( V_i)
 &=I_i(0)(\nu_\Sigma)_i
 +x_iI_i'(0)\partial_{\nu_\Sigma}\rho\\
 &=\frac{2(L_i-1)}{A}(\nu_\Sigma)_i.
\end{align*}
Since the outward normal to $E$ on its inner boundary is
$\nu_E=-\nu_\Sigma$, the function
\[
 W_i=-\frac{A}{2(1-L_i)}V_i
\]
belongs to $\mathcal D^{1,2}(E)$ and satisfies
\begin{equation} 
 \partial_{\nu_E}W_i=(\nu_E)_i,
 \qquad
 \left.W_i\right|_{\partial\Sigma}
 =-\frac{L_i}{1-L_i}x_i.
 \label{boundary value of Wi} 
\end{equation}
By  the uniqueness of the solution to the problem \eqref{eq:exterior-problem}, we obtain that $W_i$ is the solution corresponding to $q=e_i$.}

{ Next, we compute the energy $T$. Choose $R>0$ sufficiently large such that the ball $B_R$ contains $\Sigma$. We obtain that
\begin{align*}
 \int_{E\cap B_R}\nabla W_i\cdot\nabla W_j\,\dd x
 ={}&\int_{\partial\Sigma}W_i\partial_{\nu_E}W_j\,\dd S
 +\int_{\partial B_R}W_i\partial_rW_j\,\dd S.
\end{align*}
By \eqref{estimate of Vi}, the second integral on the right-hand side above is $O(R^{-n})$ and hence
tends to zero as $R\to+\infty$.  On the other hand, the divergence
theorem gives
\[
 \int_{\partial\Sigma}x_i(\nu_\Sigma)_j\,\dd S
 =\int_\Sigma\partial_jx_i\,\dd x
 =\delta_{ij}|\Sigma|.
\]
Combining $\nu_E=-\nu_\Sigma$ and \eqref{boundary value of Wi}  , we find
\begin{align*}
 T(e_i,e_j)
 &=\lim_{R\to+\infty} \int_{E\cap B_R}\nabla W_i\cdot\nabla W_j\,\dd x=\int_{\partial\Sigma}W_i(\nu_E)_j\,\dd S
 =\delta_{ij}|\Sigma|\frac{L_i}{1-L_i}.
\end{align*}}

{ Finally, if  $a_n=\tau$ and $a_1=a_2=\cdots=a_{n-1}=1$, then we obtain directly from the definition of $L_i$ that
\[
 L(\tau)=L_1=L_2=\cdots=L_{n-1}  =\frac\tau2\int_0^\infty
 \frac{\dd s}{(1+s)^{(n+1)/2}(\tau^2+s)^{1/2}}.
\]
Moreover, for $s>0$ and $0<\tau<1$,
\[
 \frac{1}{(1+s)^{(n+1)/2}(\tau^2+s)^{1/2}}
 \leq\frac{1}{(1+s)^{(n+1)/2}s^{1/2}}.
\]
The right-hand side is integrable on $(0,\infty)$ and independent of $\tau$. Consequently,
\[
 0<L(\tau)\leq C_n\tau\longrightarrow0\qquad(\tau\to0^+),
\]
which proves the asserted uniform $O(\tau)$ estimate.}
\end{proof}

\subsection{The first positive eigenvalue of a large spherical cap}

\begin{lemma}
\label{lem:highdim-large-cap}
Let $n\ge3$, and let $D_R\subset\mathbb S^n$ be the spherical cap
centered at the north pole with radius $R\in(0,\pi)$.  Its first
positive Neumann eigenvalue $\lambda(R)=\mu_2(D_R)$ has multiplicity
$n$, and its eigenspace is spanned by
$f_R(r)\theta_i$, $i=1,\dots,n$, where
$\theta\in\mathbb S^{n-1}$.  Moreover, $f_R>0$ on $(0,R]$ and may be
normalized by $f_R'(0)=1$.  As $R\uparrow\pi$,
\begin{equation}\label{eq:highdim-cap-limit}
 \lambda(R)\longrightarrow n,\qquad
 f_R\longrightarrow\sin r
 \quad\hbox{on every fixed $[0,L]\subset[0,\pi)$ in the $C^2$ sense}.
\end{equation}
Consequently, if $R$ is sufficiently close to $\pi$, there exists
$r_0$ close to $\pi/2$ such that
\begin{equation}\label{eq:highdim-cap-maximum}
 f_R'(r_0)=0,\qquad f_R''(r_0)<0,
 \qquad\lambda(R)>\frac{n-1}{\sin^2r_0}.
\end{equation}
Furthermore, $\lambda$ is a $C^1$ function on $(0,\pi)$, and
\begin{equation}\label{eq:highdim-cap-radius-derivative}
 \lambda'(R)=
 \frac{\sin^{n-1}R\,f_R(R)^2}
 {\displaystyle\int_0^R f_R(r)^2\sin^{n-1}r\,\dd r}
 \left(\frac{n-1}{\sin^2R}-\lambda(R)\right)>0
\end{equation}
for all $R$ sufficiently close to $\pi$.
\end{lemma}

\begin{proof}
{\emph{Step 1: Identification of the first positive eigenspace.}}
After separation of the angular variables, the radial operator
corresponding to the $\ell$th spherical-harmonic frequency is
\[
 -g''-(n-1)\cot r\,g'
 +\frac{\ell(\ell+n-2)}{\sin^2r}g,
\]
with the regularity condition at the north pole and the Neumann
condition at $r=R$.  
{
For $\ell\geq1$, the radial quadratic forms have the common domain
\[
 \mathcal V_R=
 \left\{g\in H^1_{\mathrm{loc}}(0,R):
 \int_0^R\left(|g'|^2+|g|^2+\frac{|g|^2}{\sin^2r}\right)
             \sin^{n-1}r\,\mathrm{d}r<\infty\right\}.
\]
Their quadratic forms are
\[
 \mathcal Q_\ell(g)=\int_0^R
 \left(|g'|^2+\frac{\ell(\ell+n-2)}{\sin^2r}|g|^2\right)
 \sin^{n-1}r\,\mathrm{d}r.
\]
For $\ell\geq2$, put
$d_\ell=\ell(\ell+n-2)-(n-1)>0$. Then
\begin{align*}
 \mathcal Q_\ell(g)-\mathcal Q_1(g)
 &=d_\ell\int_0^R|g|^2\sin^{n-3}r\,\mathrm{d}r\\
 &\geq d_\ell\int_0^R|g|^2\sin^{n-1}r\,\mathrm{d}r.
\end{align*}
Taking infima, the lowest eigenvalue for frequency $\ell$ exceeds
that for frequency one by at least $d_\ell$.
}

If $g$ is a nonconstant radial Neumann eigenfunction for $\ell=0$,
differentiating its equation shows that $g'$ satisfies the equation for
$\ell=1$, together with $g'(0)=g'(R)=0$.  Its eigenvalue is therefore
not smaller than the lowest $\ell=1$ eigenvalue with a Dirichlet
condition at the outer endpoint.  The latter is strictly larger than
the lowest eigenvalue with a Neumann condition there.  This strictness
follows from variational inclusion and uniqueness for the
one-dimensional equation: if the two lowest eigenvalues were equal,
the Dirichlet ground state would also satisfy the outer Neumann
condition, and hence would have both zero value and zero derivative
there, forcing it to vanish identically.

It follows that the full first positive eigenspace is the product of the
simple radial ground state for $\ell=1$ and the $n$-dimensional space of
first spherical harmonics.  The radial ground state may be chosen
positive and satisfies
\begin{equation}\label{eq:highdim-cap-ode}
 -(\sin^{n-1}r\,f_R')'
 +\frac{n-1}{\sin^2r}\sin^{n-1}r\,f_R
 =\lambda(R)\sin^{n-1}r\,f_R,
 \qquad f_R(0)=0,\quad f_R'(R)=0.
\end{equation}
The regular solution at the north pole satisfies
$f_R(r)=ar+O(r^3)$ with $a>0$, and we normalize it by taking $a=1$.

{\emph{Step 2: Spectral limit as the cap exhausts the sphere.}}
Write $\varepsilon=\pi-R$.  Then $D_R$ is the whole sphere with a small
cap of radius $\varepsilon$ about the south pole removed.  The
transverse coordinate $x_1$ has zero mean on $D_R$, and its Rayleigh
quotient converges to its value $n$ on the whole sphere.  Thus
$\limsup_{R\uparrow\pi}\lambda(R)\le n$.

For the reverse inequality, denote the south pole by $S$ and, in normal
coordinates centered there, set
$A_\varepsilon=B_{2\varepsilon}(S)\setminus
\overline{B_\varepsilon(S)}$.  Choose a bounded $H^1$ extension
operator on the fixed reference annulus
$B_2\setminus\overline{B_1}$.  For $u\in H^1(D_R)$, first subtract its
mean over $A_\varepsilon$, apply the rescaled extension, and then add
the mean back.  The Poincar\'e inequality and the uniform equivalence of
the metric in normal coordinates yield an extension
$\widetilde u\in H^1(\mathbb S^n)$ which equals $u$ on $D_R$ and
satisfies
\begin{align*}
 \int_{B_{2\varepsilon}(S)}|\nabla\widetilde u|^2\,\dd v
 &\le C\int_{A_\varepsilon}|\nabla u|^2\,\dd v,\\
 \int_{B_{2\varepsilon}(S)}|\widetilde u|^2\,\dd v
 &\le C\left(
 \int_{A_\varepsilon}|u|^2\,\dd v
 +\varepsilon^2\int_{A_\varepsilon}|\nabla u|^2\,\dd v
 \right),
\end{align*}
where $C$ is independent of $\varepsilon$.

Let $u_R$ be a first positive eigenfunction normalized by
$\int_{D_R}u_R=0$ and $\int_{D_R}u_R^2=1$.  The spectral upper bound
above implies that its extension $\widetilde u_R$ is bounded in
$H^1(\mathbb S^n)$.  The Sobolev inequality gives
\[
 \int_{B_\varepsilon(S)}|\widetilde u_R|^2\,\dd v
 \le |B_\varepsilon(S)|^{2/n}
 \|\widetilde u_R\|_{L^{2n/(n-2)}(\mathbb S^n)}^2
 \le C\varepsilon^2.
\]
{Choose $R_j\uparrow\pi$ such that
$\lambda(R_j)\to\underline\lambda:=\liminf_{R\uparrow\pi}\lambda(R)$.
After passing to a subsequence,
$\widetilde u_{R_j}\rightharpoonup u$ in $H^1(\mathbb S^n)$ and
$\widetilde u_{R_j}\to u$ in $L^2(\mathbb S^n)$.}
The mass estimate inside the hole ensures that
$\int_{\mathbb S^n}u=0$ and $\int_{\mathbb S^n}u^2=1$.
{For each fixed $\rho\in(0,\pi)$, the extension agrees with
$u_{R_j}$ on $\mathbb S^n\setminus\overline{B_\rho(S)}$ for all large $j$.
Weak lower semicontinuity on this fixed domain gives
\[
 \int_{\mathbb S^n\setminus\overline{B_\rho(S)}}|\nabla u|^2\,\dd v
 \le\liminf_{j\to\infty}
 \int_{\mathbb S^n\setminus\overline{B_\rho(S)}}|\nabla u_{R_j}|^2\,\dd v
 \le\underline\lambda.
\]
Letting $\rho\downarrow0$ and using the zero mean and unit mass of $u$,
we obtain
\[
 n\le\int_{\mathbb S^n}|\nabla u|^2\,\dd v
 \le\underline\lambda.
\]
Together with the upper bound, this proves $\lambda(R)\to n$.}

{\emph{Step 3: Convergence of the radial factor on compact intervals.}}
Set $h_R=f_R/\sin r$ and extend it continuously to the origin by
$h_R(0)=1$.  Equation~\eqref{eq:highdim-cap-ode} is equivalent to
\begin{equation}\label{eq:highdim-cap-volterra}
 h_R(r)=1+(n-\lambda(R))
 \int_0^r\frac1{\sin^{n+1}s}
 \int_0^s\sin^{n+1}t\,h_R(t)\,\dd t\,\dd s.
\end{equation}
Fix $L<\pi$, and denote the integral operator in
\eqref{eq:highdim-cap-volterra} by $\mathcal K_L$.  Near the origin,
the inner integral divided by $\sin^{n+1}s$ is $O(s)$, so
$\mathcal K_L$ is bounded on $C([0,L])$.  Let
$\delta_R=n-\lambda(R)$.  When $R$ is sufficiently close to $\pi$,
$|\delta_R|\|\mathcal K_L\|<1/2$, and the Neumann series gives
\[
 \|h_R-1\|_{C([0,L])}\le C_L|\delta_R|.
\]
Differentiating~\eqref{eq:highdim-cap-volterra} gives the exact formula
\[
 h_R'(r)=\delta_R\sin^{-n-1}r
 \int_0^r\sin^{n+1}t\,h_R(t)\,\dd t.
\]
Fix a small $r_*>0$.  Since $\sin r\asymp r$ on $[0,r_*]$, the
preceding identity gives
\[
 |h_R'(r)|\le C_L|\delta_R|r
 \quad(0\le r\le r_*),
 \qquad
 |h_R'(r)|\le C_L|\delta_R|
 \quad(r_*\le r\le L).
\]
In particular, $\cot r\,h_R'(r)=O(|\delta_R|)$ uniformly near the
origin.  Using
\[
 h_R''=\delta_R h_R-(n+1)\cot r\,h_R'
\]
we obtain $\|h_R''\|_{C([0,L])}\le C_L|\delta_R|$.  Hence
$h_R\to1$ in $C^2([0,L])$, and consequently
$f_R=\sin r\,h_R\to\sin r$ in $C^2([0,L])$.
{Fix $0<\delta<\pi/4$ and choose $\pi/2+\delta<L<\pi$.
For $R$ sufficiently close to $\pi$, the $C^2$ convergence gives
\[
 f_R'(\pi/2-\delta)>0,\qquad f_R'(\pi/2+\delta)<0,
 \qquad f_R''\le-\tfrac12\cos\delta<0
 \quad\hbox{on }[\pi/2-\delta,\pi/2+\delta].
\]
Thus $f_R$ has a unique critical point $r_0$ in this interval,
and it is a strict local maximum. Moreover,
$|\cos r_0|\le\|f_R'-\cos r\|_{C([0,L])}$, so $r_0\to\pi/2$.
At $r_0$, equation~\eqref{eq:highdim-cap-ode} gives
\[
 \lambda(R)-\frac{n-1}{\sin^2r_0}
 =-\frac{f_R''(r_0)}{f_R(r_0)}>0,
\]
which proves~\eqref{eq:highdim-cap-maximum}.}

{\emph{Step 4: Derivative with respect to the radius.}}
Let $f_\lambda$ be the radial solution that is regular at the north pole
and normalized by $f_\lambda'(0)=1$, and set
$v=\partial_\lambda f_\lambda$.  
{
Fix $R_*\in(0,\pi)$ and choose $L\in(R_*,\pi)$.
Writing $f_\lambda(r)=\sin r\,h_\lambda(r)$, the integral equation
corresponding to \eqref{eq:highdim-cap-volterra} is
\[
 h_\lambda=1+(n-\lambda)\mathcal K_Lh_\lambda,
 \qquad
 (\mathcal K_Lh)(r)=\int_0^r k(r,t)h(t)\,\mathrm{d}t,
\]
where
\[
 k(r,t)=\sin^{n+1}t\int_t^r\sin^{-n-1}s\,\mathrm{d}s,
 \qquad 0<t\leq r\leq L.
\]
Since $\sin s$ is comparable to $s$ on $[0,L]$,
$0\leq k(r,t)\leq C_Lt$. The kernel therefore extends continuously to
$t=0$ with value zero. Let $M_L$ be its maximum on the closed triangle.
Integration over the ordered simplex gives
\[
 \|\mathcal K_L^j\|_{C([0,L])\to C([0,L])}
 \leq\frac{(M_LL)^j}{j!},\qquad j\geq0.
\]
Hence the series
\[
 h_\lambda=\sum_{j=0}^{\infty}(n-\lambda)^j\mathcal K_L^j1
\]
and its derivative with respect to $\lambda$ converge uniformly in
$C([0,L])$ on bounded parameter intervals.
Differentiation in $r$ gives
\[
 h_\lambda'(r)=(n-\lambda)\sin^{-n-1}r
          \int_0^r\sin^{n+1}t\,h_\lambda(t)\,\mathrm{d}t.
\]
The identity and its parameter derivative show that
$F(\lambda,R)=f_\lambda'(R)$ is $C^1$ near
$(\lambda(R_*),R_*)$. As $r\downarrow0$,
\[
 h_\lambda(r)=1+O(r^2),\qquad
 \partial_\lambda h_\lambda(r)=O(r^2),\qquad
 \partial_r\partial_\lambda h_\lambda(r)=O(r).
\]
In particular, $v(r)=O(r^3)$ and $v'(r)=O(r^2)$.
}
At $\lambda=\lambda(R)$, the Wronskian identity applied to the equation
and its $\lambda$ derivative gives
\[
 -\sin^{n-1}R\,f_R(R)\,v'(R)
 =\int_0^R f_R(r)^2\sin^{n-1}r\,\dd r>0.
\]
The boundary term at the origin vanishes because the regular solution is
$O(r)$ there, whereas $v=O(r^3)$.  In particular, $v'(R)\ne0$.
Set $F(\lambda,R)=f_\lambda'(R)$.  The identity above shows that
$\partial_\lambda F(\lambda(R),R)=v'(R)\ne0$.  Thus the implicit
function theorem gives a unique local $C^1$ branch of roots near each
$R\in(0,\pi)$.  The radial ground state for $\ell=1$ is simple and
strictly positive on $(0,R]$.  After shrinking the parameter
neighborhood, the regular solution associated with this local root
remains positive and therefore remains the $\ell=1$ ground state rather
than a higher radial branch.  Uniqueness identifies the root branches
on overlapping neighborhoods, so they combine to form a $C^1$ function
$\lambda(R)$ on $(0,\pi)$.

Differentiating $F(\lambda(R),R)=0$ gives
$\lambda'(R)=-f_R''(R)/v'(R)$.  The identity in
\eqref{eq:highdim-cap-radius-derivative} then follows from
$f_R'(R)=0$ and the radial equation.  Finally,
$\lambda(R)\to n$, whereas $(n-1)/\sin^2R\to+\infty$.  Hence this
derivative is strictly positive when $R$ is sufficiently close to
$\pi$.
\end{proof}

\subsection{Construction of the counterexample}

\begin{proof}[Proof of Theorem~\ref{thm:highdim-counterexample}]
By Lemma~\ref{lem:highdim-large-cap}, fix $R$ sufficiently close to
$\pi$ and $r_0\in(0,R)$ such that the positive radial factor
$f=f_R$ satisfies
\begin{equation}\label{eq:highdim-cap-interior-strict-maximum}
 f'(r_0)=0,\qquad f''(r_0)<0,\qquad
 \lambda(R)>\frac{n-1}{\sin^2r_0},\qquad \lambda'(R)>0.
\end{equation}
Let $\kappa>0$ be the normalization constant such that
$u_i(r,\theta)=\kappa f(r)\theta_i$, $i=1,\ldots,n$, form an
$L^2$-orthonormal basis of the first positive eigenspace of the spherical cap $D_R$.

At the latitude $r_0$, choose the $2n$ points
$p_{\pm i}=(r_0,\pm e_i)$ for $1\leq i\leq n$.  In each tangent space, take
$\partial_r$ as the short axis of the oblate ellipsoid and the remaining
$n-1$ angular directions as its long axes.  In the corresponding
orthonormal coordinates
$y=(y_r,y_{\mathrm{ang}})\in\mathbb R\times\mathbb R^{n-1}$, set
\[
 \Sigma_\tau=\left\{y:\frac{y_r^2}{\tau^2}+|y_{\mathrm{ang}}|^2<1\right\},
\]
where the short semiaxis has length $\tau$, while all remaining semiaxes have
length $1$.  First fix $\tau>0$ sufficiently small that
\begin{equation}\label{eq:highdim-oblate-positive-margin}
 \lambda(R)>\frac{n-1}{(1-L(\tau))\sin^2r_0},
\end{equation}
where $L(\tau)$ is given by Lemma \ref{lem:highdim-oblate-polarization}. This is possible by the fact that
$L(\tau)\to0$ as $\tau\to0^+$.

{
Next, we define the perforated domain by
\[
 \Om_\varepsilon:=D_R\setminus\bigcup_{i,\pm}
 \overline{\exp_{p_{\pm i}}(\varepsilon\Sigma_\tau)}.
\]
 Here $R$ and $\tau$ are fixed and only $\varepsilon$ is allowed to tend to 0.  Thus the holes are pairwise
disjoint, remain away from the outer boundary, and preserve the original
spectral gap. More precisely,  at the centers of the holes, we have}
\begin{equation} 
\label{sums of function values} 
 \sum_{i,\pm}u_j(p_{\pm i})u_k(p_{\pm i})
 =2\kappa^2f(r_0)^2\delta_{jk}.
\end{equation}
{
Since $f'(r_0)=0$, all radial derivatives vanish. The identity
\[
 \left\langle\nabla_{\mathbb S^{n-1}}\theta_j,
 \nabla_{\mathbb S^{n-1}}\theta_k\right\rangle
 =\delta_{jk}-\theta_j\theta_k
\]
and the angular metric factor $\sin^{-2}r_0$ give
}
\begin{equation}
\label{sums of gradients} 
 \sum_{i,\pm}\langle\nabla u_j(p_{\pm i}),\nabla u_k(p_{\pm i})\rangle
 =\frac{2(n-1)\kappa^2f(r_0)^2}{\sin^2r_0}\delta_{jk}.
\end{equation}
Write $q_j^{\pm i}=\nabla u_j(p_{\pm i})$.  Since $f'(r_0)=0$, each
$q_j^{\pm i}$ has zero radial component and therefore belongs to the
long-axis subspace of the oblate ellipsoid.
Lemma~\ref{lem:highdim-oblate-polarization} and bilinear polarization
give
\begin{equation}
\label{polarization energies} 
 T_{\pm i}(q_j^{\pm i},q_k^{\pm i})
 =|\Sigma_\tau|\frac{L(\tau)}{1-L(\tau)}
 \langle q_j^{\pm i},q_k^{\pm i}\rangle.
\end{equation}
Substituting \eqref{sums of function values}, \eqref{sums of gradients} and \eqref{polarization energies} into
\eqref{eq:highdim-manifold-hole-matrix}, we obtain
\begin{align*}
 (\mathbf B)_{jk}
 ={}&2|\Sigma_\tau|\kappa^2f(r_0)^2
 \left[
 \lambda(R)-\frac{n-1}{\sin^2r_0}
 -\frac{(n-1)L(\tau)}
 {(1-L(\tau))\sin^2r_0}
 \right]\delta_{jk}\\
 ={}&2|\Sigma_\tau|\kappa^2f(r_0)^2
 \left[
 \lambda(R)-\frac{n-1}
 {(1-L(\tau))\sin^2r_0}
 \right]\delta_{jk}.
\end{align*}
Thus $\mathbf B=cI_n$, where
\begin{equation}\label{eq:highdim-oblate-positive-cluster}
 c=2|\Sigma_\tau|\kappa^2f(r_0)^2
 \left[\lambda(R)-\frac{n-1}{(1-L(\tau))\sin^2r_0}\right]>0.
\end{equation}
Consequently, by Lemma \ref{lem:highdim-manifold-small-holes}, we obtain
\[
 \mu_j(\Om_\varepsilon)=\lambda(R)+c\varepsilon^n+o(\varepsilon^n),
 \qquad j=2,\,\cdots,\,n+1.
\]
{
Since $c>0$, we have $\mu_j\left(\Omega_\varepsilon\right)> \lambda(R)$ for $j=2,\,\cdots,\,n+1$ and sufficiently small $\varepsilon$.
}

Note that
\[
 |D_R|-|\Om_\varepsilon|
 =2n|\Sigma_\tau|\varepsilon^n+o(\varepsilon^n)>0.
\]
Define the spherical cap volume function by
\[
 V(s)=|\mathbb S^{n-1}|\int_0^s\sin^{n-1}t\,\dd t.
\]
Let $B_\varepsilon$ be the spherical cap having the same volume as $\Omega_\varepsilon$ with radius $R_\varepsilon$.  Since $V(R_\varepsilon)=|\Om_\varepsilon|$ and $V'(R)=|\mathbb S^{n-1}|\sin^{n-1}R>0$, the Taylor expansion of {$V$ at $R$} gives
\[
 R-R_\varepsilon
 =\frac{2n|\Sigma_\tau|}
 {|\mathbb S^{n-1}|\sin^{n-1}R}\varepsilon^n
 +o(\varepsilon^n)>0.
\]
In particular, $R_\varepsilon<R$ and $R_\varepsilon\to R$.
Since $\lambda'(R)>0$ and $\lambda'$ is continuous, $\lambda'>0$ in a
neighborhood of $R$.  Therefore, for all sufficiently small
$\varepsilon$,
\[
 \mu_j(\Om_\varepsilon)\ge\mu_2(\Om_\varepsilon)>\lambda(R)>
 \lambda(R_\varepsilon)=\mu_2(B_\varepsilon),\qquad j=2,\ldots,n+1.
\]
It follows that
\[
 \sum_{j=2}^{n+1}\frac1{\mu_j(\Om_\varepsilon)}
 <\frac n{\mu_2(B_\varepsilon)}.
\]
Geometrically, the closure of each hole is diffeomorphic to a closed
$n$-ball and lies in an interior coordinate neighborhood disjoint from
all the others.  Hence $\Om_\varepsilon$ is connected and has
$2n+1$ boundary components.  We verify the simple connectedness of $\Omega_\varepsilon$ by removing the balls from $D_R$ successively.   { At one step, let $D^+$ be the domain before removal and let $K\Subset D^+$ be the closed ball to be removed. Let $D^-=D^+\setminus K$. Choose a sufficiently small open collar $C$ of $\partial K$ in $D^+$, and define 
\begin{equation*}
    U=D^-,\qquad V=\operatorname{int}(K)\cup C.
\end{equation*}
Then $U$ and $V$ are open, path connected subsets of $D^+$, and $D^+=U\cup V$. The collar $C$ can be chosen so that $V$ deformation retracts onto $K$, whereas $U\cap V=C\setminus K$ deformation retracts onto $\partial K\cong \mathbb{S}^{n-1}$. Thus
\begin{equation*}
\pi_1(V)=0,\qquad\pi_1(U\cap V)=\pi_1(\mathbb{S} ^{n-1})=0,
\end{equation*}
where the second equality uses \(n\geq3\). Applying the Seifert-van Kampen theorem to this open cover gives
\begin{equation*}
\pi_1(D^+)\cong\pi_1(U)=\pi_1(D^-).
\end{equation*}
Repeating the argument for all $2n$ holes yields
\begin{equation*} 
\pi_1(\Omega_\varepsilon)\cong\pi_1(D_R)=0,
\end{equation*}
since $D_R$ is a geodesic ball with $R<\pi$ and is therefore diffeomorphic to an open $n$-ball.
  { Finally, since $R>\pi/2$, we have $|D_R|>|\mathbb S^n|/2$.
The convergence $|\Om_\varepsilon|\to|D_R|$ implies
$|\Om_\varepsilon|>|\mathbb S^n|/2$ for all sufficiently small
$\varepsilon$.} Fix such an $\varepsilon$ and set
$\Om=\Om_\varepsilon$.  This is the required counterexample.}
\end{proof}

{
 \noindent\textbf{Acknowledgement:}  M. Liu  is funded by NSFC (12601191) and Zhejiang Provincial Natural Science Foundation of China (LQN25A010007) and the Fundamental Research Funds for the Provincial Universities of Zhejiang (GK259909299001-029). W.
Zou is funded by National Key R\&D Program of China (Grant 2023YFA1010001) and NSFC (12171265).

\subsection*{Declarations}

\noindent\textbf{Availability of data and materials.} No datasets were generated or analysed during the current study.

\smallskip
\noindent\textbf{Conflict of interest.} The authors declare no conflict of interest.

\smallskip
\noindent\textbf{Use of artificial intelligence.} AI-assisted tools were used during the preparation of this manuscript to assist with language editing, presentation, and checks of the mathematical exposition. The authors take full responsibility for the content of the manuscript, including all mathematical arguments and conclusions.}

\endgroup


\begin{thebibliography}{99}
\small

\bibitem{BallZarnescu2017} {J.~M. Ball, A. Zarnescu}, Partial regularity and smooth topology-preserving approximations of rough domains, Calc. Var. Partial Differential Equations {\bf 56} (2017), no.~1, Paper No. 13, 32 pp.

\bibitem{Bandle1972} C. Bandle, Isoperimetric inequality for some eigenvalues of an inhomogeneous, free membrane, SIAM J. Appl. Math. {\bf 22} (1972), 142--147.

\bibitem{BucurLaugesenMartinetNahon2025}
D.~Bucur, R.~S. Laugesen, E.~Martinet, M.~Nahon, Spherical caps do not always maximize Neumann eigenvalues on the sphere, Geom. Funct. Anal. {\bf 35} (2025), no.~5, 1313--1345.

\bibitem{cifea1980} I. Chavel, E.~A. Feldman, Isoperimetric inequalities on curved surfaces, Adv. in Math. {\bf 37} (1980), no.~2, 83--98.

\begingroup
\bibitem{ChenYang2026}
D.~Chen, C.~Yang,
\newblock Sharp harmonic-mean inequalities for Neumann and Aharonov--Bohm spectra,
\newblock arXiv:2609.24488v1 (2026).
\par\endgroup

\bibitem{Chenais1975} D. Chenais, On the existence of a solution in a domain identification problem, J. Math. Anal. Appl. {\bf 52} (1975), no.~2, 189--219.

\bibitem{ColboisProvenzanoSavo2022} B. Colbois, L. Provenzano, A. Savo, Isoperimetric inequalities for the magnetic Neumann and Steklov problems with Aharonov-Bohm magnetic potential, J. Geom. Anal. {\bf 32} (2022), no.~11, Paper No. 285, 38 pp.

\bibitem{deb1995} E.~B. Davies, {\it Spectral theory and differential operators}, Cambridge Studies in Advanced Mathematics, 42, Cambridge Univ. Press, Cambridge, 1995.

\bibitem{DiFratta2016} G. Di~Fratta, The Newtonian potential and the demagnetizing factors of the general ellipsoid, Proc. A {\bf 472} (2016), no.~2190, 20160197, 7 pp.

\begingroup
\par\endgroup

\bibitem{FelliLiveraniOgnibene2026}
V.~Felli, L.~Liverani, and R.~Ognibene,
\newblock On the splitting of Neumann eigenvalues in perforated domains,
\newblock arXiv:2601.13129v1 (2026).

\bibitem{LangfordLaugesen2023} J.~J. Langford, R.~S. Laugesen, Maximizers beyond the hemisphere for the second Neumann eigenvalue, Math. Ann. {\bf 386} (2023), no.~3-4, 2255--2281.

\bibitem{LangfordLaugesenScaling2023} J.~J. Langford, R.~S. Laugesen, Scaling inequalities for spherical and hyperbolic eigenvalues, J. Spectr. Theory {\bf 13} (2023), no.~1, 263--296.

\bibitem{MichettiProvenzanoSavo2026}
M.~Michetti, L.~Provenzano, A.~Savo,
\newblock Isoperimetric inequalities and sharp upper bounds for
Aharonov--Bohm eigenvalues on surfaces,
\newblock arXiv:2604.11718v2 (2026).

\bibitem{pc} C. Pommerenke, {\it Boundary behaviour of conformal maps}, Grundlehren der mathematischen Wissenschaften, 299, Springer, Berlin, 1992.

\bibitem{ProvenzanoSavo2026}
L.~Provenzano and A.~Savo,
\newblock On the isoperimetric inequality for the first positive
Neumann eigenvalue on the sphere,
\newblock arXiv:2603.04236v4 (2026).

\end{thebibliography}
\end{document}